\documentclass[12pt,a4paper]{article}
\usepackage{amsmath,amssymb,amsthm,graphicx,color,bbm,array}
\usepackage[left=1in,right=1in,top=1in,bottom=1in]{geometry}
\usepackage{cite}
\usepackage[british]{babel}
\usepackage{hyperref} 
\usepackage{enumitem}
\usepackage{tikz,pgfplots}
\usetikzlibrary{decorations.markings,decorations.pathmorphing,shapes}
\usepackage{caption}
\usepackage{subcaption}

\pgfplotsset{ticks=none}
\pgfplotsset{compat=1.18}

\usepackage{xcolor}

\newcounter{sarrow}

\hypersetup{
    colorlinks=false,
    pdfborder={0 0 0},
}

\newtheorem{theorem}{Theorem}[section]
\newtheorem{corollary}[theorem]{Corollary}
\newtheorem{proposition}[theorem]{Proposition}
\newtheorem{lemma}[theorem]{Lemma}
 \newtheorem{conjecture}[theorem]{Conjecture}

\numberwithin{equation}{section}

\theoremstyle{definition}

\newtheorem{examplex}[theorem]{Example}
 \newtheorem*{problem}{Open problem}

\newenvironment{example}
  {\pushQED{\qed}\examplex}
  {\popQED\endexamplex}

\theoremstyle{remark}
\newtheorem{remark}[theorem]{Remark}
\newtheorem{remarks}[theorem]{Remarks}
\newtheorem*{remark*}{Remark}

 \allowdisplaybreaks

\newcommand{\1}[1]{{\mathbbm{1}\mkern -1.5mu}{\{#1\}}}

\newcommand{\R}{{\mathbb R}}

\newcommand{\Z}{{\mathbb Z}}
\newcommand{\N}{{\mathbb N}}

\newcommand{\ZP}{{\mathbb Z}_+}
\newcommand{\RP}{{\mathbb R}_+}

\DeclareMathOperator{\Exp}{\mathbb{E}}
\let\Pr\relax
\DeclareMathOperator{\Pr}{\mathbb{P}}

\DeclareMathOperator{\sign}{sgn}

\newcommand{\eps}{\varepsilon}

\newcommand{\ud}{{\mathrm d}}

\newcommand{\cC}{{\mathcal C}}

\newcommand{\cF}{{\mathcal F}}
\newcommand{\cG}{{\mathcal G}}

\newcommand{\hW}{{\widehat W}}
\newcommand{\hT}{{\widehat T}}

\newcommand{\as}{\ \text{a.s.}}

\newcommand{\bigmid}{\; \bigl| \;}
\newcommand{\Bigmid}{\; \Bigl| \;}

\makeatletter
\def\namedlabel#1#2{\begingroup  
    (#2)%
    \def\@currentlabel{#2}%
    \phantomsection\label{#1}\endgroup
}
\makeatother

\newlist{myenumi}{enumerate}{10}
\setlist[myenumi]{leftmargin=0pt, labelindent=\parindent, listparindent=\parindent, labelwidth=0pt, itemindent=!, itemsep=1pt, parsep=4pt}

\newlist{thmenumi}{enumerate}{10}
\setlist[thmenumi]{leftmargin=0pt, labelindent=\parindent, listparindent=\parindent, labelwidth=0pt, itemindent=!}

\title{Superdiffusive planar random walks with polynomial space-time drifts}
\author{Conrado da Costa\footnote{\raggedright Department of Mathematical Sciences, Durham University, Upper Mountjoy, Durham, DH1 3LE, UK. Email:~\href{mailto:conrado.da-costa@durham.ac.uk}{\texttt{conrado.da-costa@durham.ac.uk}}, \href{mailto:mikhail.menshikov@durham.ac.uk}{\texttt{mikhail.menshikov@durham.ac.uk}}, \href{mailto:andrew.wade@durham.ac.uk}{\texttt{andrew.wade@durham.ac.uk}}.}  \and Mikhail Menshikov\footnotemark[1] \and Vadim Shcherbakov\footnote{Department of Mathematics, Royal Holloway, University of London, McCrea Building, Egham, Surrey, TW20 0EX, UK. Email:~\href{mailto:vadim.shcherbakov@rhul.ac.uk}{\texttt{vadim.shcherbakov@rhul.ac.uk}}.} \and Andrew Wade\footnotemark[1]}

\begin{document}

\maketitle

\begin{abstract}
We quantify superdiffusive transience for a two-dimensional random walk
in which the vertical coordinate is a martingale and the horizontal coordinate has a positive drift that is a polynomial function of the individual coordinates and of the present time. We describe how the model was motivated through an heuristic connection to a self-interacting, planar random walk which interacts with its own centre of mass via an excluded-volume mechanism, and is conjectured to be superdiffusive with a scale exponent~$3/4$. The self-interacting process originated in discussions with Francis Comets.
\end{abstract}

\medskip


\medskip


\section{Introduction}
\label{sec:intro}

To motivate the model that we formulate and study,
we first describe a planar, self-interacting random walk. 
Our random walk will be a process $W = (W_n, n \in \ZP)$ with $W_n \in \R^2$, indexed by discrete time $n \in \ZP: = \{0,1,2,\ldots\}$, 
started from $W_0 =0$,
whose increment law is determined by the current location $W_n$ of the walk and the current \emph{centre of mass} of the previous trajectory, $G_n \in \R^2$, defined by $G_0 := 0$ and
\begin{equation}
    \label{eq:G-def}
 G_n := \frac{1}{n}
\sum_{k=1}^n  W_k , \text{ for } n \in \N := \{1,2,3,\ldots\}. \end{equation}
At time~$n$, given the locations $G_n$ and $W_n$ in $\R^2$,
we take the new location $W_{n+1}$
to be uniformly distributed on the arc of the unit-radius circle, centred at~$W_n$,
excluding the triangle with vertices $0, G_n, W_n$: see Figure~\ref{fig:flory-increment}
for a schematic and simulation, and Section~\ref{sec:barycentre} below for a more formal description.

\begin{figure}[!h]
\begin{center}
\scalebox{0.95}{\begin{tikzpicture}[domain=0:10, scale = 1.2,decoration={
    markings,
    mark=at position 0.6 with {\arrow[scale=2]{>}}}]
\draw ({5+1.0*cos(180)},{0+1.0*sin(180)}) arc (180:153:1.0);
\draw[double] ({5+0.8*cos(153)},{0+0.8*sin(153)}) arc (153:{207-387}:0.8);
\filldraw (0,0) circle (2pt);
\filldraw (3,1) circle (2pt);
\filldraw (5,0) circle (2pt);
\draw (0,0) -- (5,0);
\draw (3,1) -- (5,0);
\node at (-0.2,0)       {$0$};
\node at (4.9,-0.3)       {$W_{n}$};
\node at (3,1.3)       {$G_{n}$};
\node at (3.5,0.3)       {$2 \beta_{n}$};
\node at (5.9,1.3)       {$W_{n+1}$};
\draw[->,
line join=round,
decorate, decoration={
    zigzag,
    segment length=4,
    amplitude=.9,post=lineto,
    post length=2pt
}] (5.8,1.1) -- (5.5,0.76);
\end{tikzpicture}}\hfill 
\scalebox{-1}[1]{\includegraphics[angle=85,width=0.48\textwidth]{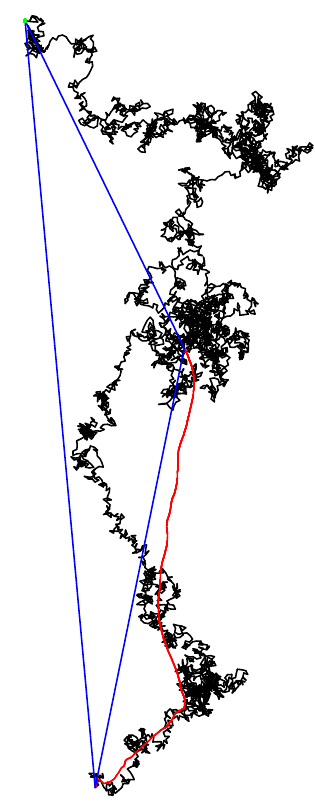}}
\end{center}
\caption{\label{fig:flory-increment} 
\emph{Left:}
The angle at vertex $W_n$ of the triangle with vertices $0, G_n, W_n$
is denoted by~$2 \beta_n$. Given $G_n, W_n$,
the distribution of~$W_{n+1}$ is uniform on the unit disc centred at~$W_n$
but excluding the arc of angle $2\beta_n$ as indicated. \emph{Right:}
A simulated trajectory (not to scale with the unit disk on the left), 
plus the centre of mass trajectory (in red) and
the excluded triangle for the next step (in blue).}
\end{figure}
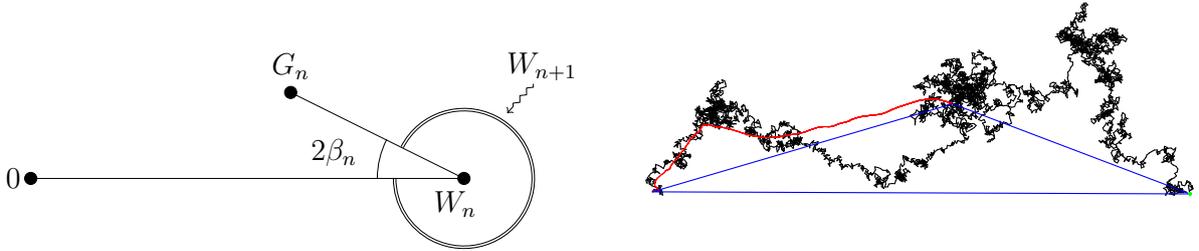

Rigorous analysis of this model remains an open problem,
but both simulation evidence (see Figure~\ref{fig:flory-sims}), and an heuristic argument (see Figure~\ref{fig:flory-heuristic} and surrounding discussion)
that links the model to a special case of Theorem~\ref{thm:main-thm} below,
suggest that, a.s.,
\begin{equation}
\label{eq:flory-scaling}
    \lim_{n \to \infty} \frac{\log \| W_n \|}{\log n} = \frac{3}{4} , \text{ and }
    \lim_{n \to \infty} \frac{W_n}{\| W_n \|} = \Theta ,
\end{equation}
where $\Theta$ is a uniform random unit vector (a limiting direction for the walk).
A scale exponent of $3/4 > 1/2$
evidences \emph{anomalous diffusion}~\cite{mjcb,oflv}.
Furthermore, the specific exponent~$3/4$
is famous in the connection of planar self-interacting walks as the \emph{Flory exponent}.
The end-to-end distance of a uniformly-sampled $n$-step self-avoiding random walk (SAW) 
on the planar lattice $\Z^2$ is expected to scale like $n^{3/4}$,
a prediction that goes back to Flory's work in the physical chemistry of
polymers in solution~\cite{flory,flory2};
see e.g.~\cite{bn,hughes,hk,lsw,ms} for further background on the celebrated SAW model.
  Flory's prediction has since been reproduced
  by several separate  arguments
  (see Chapter 2 of \cite{ms}, Chapters 7--11 of
  \cite{rg}, Chapter 15 of \cite{kleinert}, and  \cite{alkhimov}, for example)
  and is now interpreted in the context of the famous conjecture that
	the scaling limit of SAW is Schramm--Loewner 
 evolution with parameter $8/3$~\cite{lsw}.

Unlike the $(W_n, G_n)$ process, which is Markov, SAW is not a natural stochastic process, 
and it seems unlikely that there is a deep connection underlying the appearance of the $3/4$ exponent in the (conjectural) behaviour of the two models;
we do not expect an SLE limit, for instance, and, indeed, the heuristic in Section~\ref{sec:barycentre} below
can be pursued to suggest a different scaling limit. 
Nevertheless, the $(W_n, G_n)$ process was inspired by
discussions with Francis Comets in the context of self-interacting random-walk
models of polymer chains, and so the appearance of the Flory exponent is attractive.
   The basic idea is that
   the steps of the
  walk represent similar segments (such as monomers)
  in the molecule.
The inspiration for the self-interaction mechanism for the $(W_n, G_n)$ process is the \emph{excluded volume effect}~\cite{bn,flory2} from polymer physics,
which says that no two monomers can occupy the same physical space,
but we impose this condition in a coarse sense, mediated by the barycentre (centre of mass) of the previous trajectory.
One may also view the $(W_n, G_n)$ 
process in the context of random walks interacting with their past
occupation measures, which is a topic of ongoing interest;
  models of a similar flavour, some of which having been studied by Francis Comets, include~\cite{abv,BFV10,cmw,cmvw}. 

\begin{figure}[!h]
\begin{center}
\scalebox{0.98}{\begin{tikzpicture}[domain=0:10, scale = 1.2,decoration={
    markings,
    mark=at position 0.6 with {\arrow[scale=2]{>}}}]
\draw[double] ({5+0.8*cos(153)},{0+0.8*sin(153)}) arc (153:{900}:0.8);
\filldraw (0,0) circle (2pt);
\filldraw (3,0) circle (2pt);
\filldraw (5,0) circle (2pt);
\draw (0,0) -- (5,0);
\node at (-0.2,0)       {$0$};
\node at (5.1,-0.3)       {$W_n$};
\node at (3,0.3)       {$G_n$};
\node at (3,-1)       {\phantom{x}};
\end{tikzpicture}}
\quad
\scalebox{0.98}{\begin{tikzpicture}[domain=0:10, scale = 1.2,decoration={
    markings,
    mark=at position 0.6 with {\arrow[scale=2]{>}}}]
\draw ({5+1.0*cos(169)},{-1+1.0*sin(169)}) arc (169:156:1.0);
\draw[double] ({5+0.8*cos(156)},{-1+0.8*sin(156)}) arc (156:{207-398}:0.8);
\filldraw (0,0) circle (2pt);
\filldraw (3,-0.1) circle (2pt);
\filldraw (5,-1) circle (2pt);
\draw (0,0) -- (5,-1);
\draw (3,-0.1) -- (5,-1);
\node at (-0.2,0)       {$0$};
\node at (4.9,-1.3)       {$W_{(1+\eps)n}$};
\node at (3,0.3)       {$G_{(1+\eps)n}$};
\end{tikzpicture}}
\end{center}
\caption{\label{fig:flory-heuristic} Suppose that $\| W_n \| \approx n^\chi$ for some $\chi > 1/2$. We expect that a typical configuration has $0, G_n$, and $W_n$ roughly aligned at scale~$n^\chi$ (\emph{left panel}). At this point, the drift of $W_n$ is close to zero, and over a time period of about $\eps n$, $W$ will move roughly diffusively, while $G_n$ will move only a little. Hence a typical picture at time $(1+\eps) n$ would have $\beta_{(1+\eps)n} \approx n^{(1/2)-\chi}$ (\emph{right panel}).}
\end{figure}
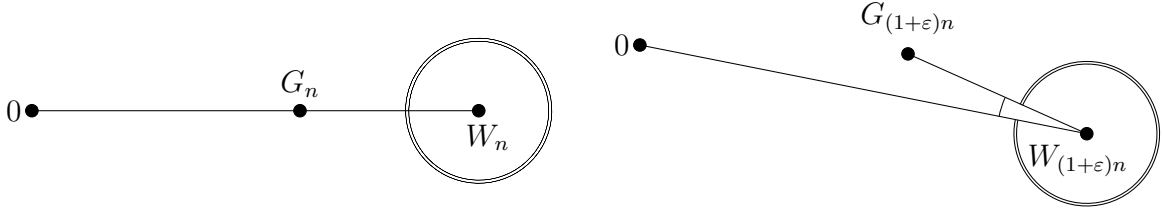

We sketch an heuristic argument 
for the scaling in~\eqref{eq:flory-scaling};
we will in Section~\ref{sec:barycentre}
give a slightly more developed version of the heuristic, which links
the $(W_n, G_n)$
process to the model that we study in the bulk of this paper (see Example~\ref{ex:34} below).
We expect
that the typical configuration is for
$0$, $G_n$, and $W_n$ to lie on roughly the same line,
in that order (see Figure~\ref{fig:flory-heuristic}, left panel). 
Suppose that $\|W_n\| \approx n^\chi$
 for some $\chi$ and that $G_n$ is about half way  between
 $0$ and $W_n$ (under a strong version of~\eqref{eq:flory-scaling}, $G_n$ would typically lie about $4/7$ of the way from $0$ to $W_n$).
When $0, G_n$, and $W_n$ are in alignment, 
the mean drift of~$W_n$ is zero, because the increment is symmetric. 
Thus on a moderate time-scale, one expects the location of $W_n$
to have wandered distance of order $n^{1/2}$ in the perpendicular
 direction, while $G_n$ will have moved much less because of its more stable dynamics. 
 At that point, $G_n$ and $W_n$ are still on scale about $n^\chi$, but now the angle, $\beta_n$,
 at $W_n$ in the triangle formed by $0, W_n, G_n$,
 will be of size about $n^{(1/2)-\chi}$ (see Figure~\ref{fig:flory-heuristic}, right panel). The size of the angle $\beta_n$
 is also the order of the magnitude of the
 drift of the walk. A consistency argument demands that $\sum_{m=1}^n m^{(1/2)-\chi} \approx n^\chi$,
 i.e., $(3/2) - \chi = \chi$,  which gives $\chi = 3/4$.

\begin{figure}[!h]
\begin{center}
\includegraphics[width=0.4\textwidth]{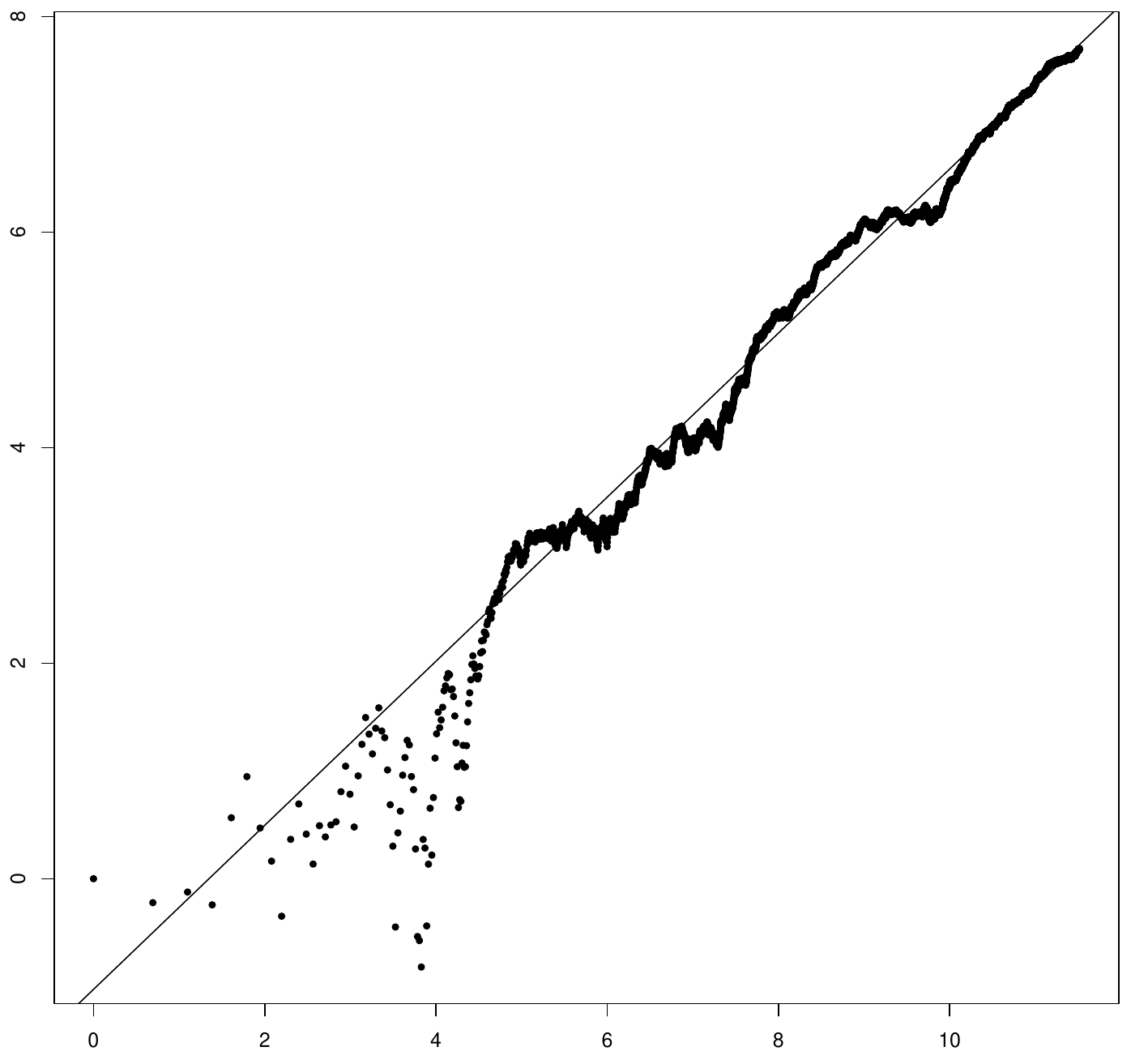}
\quad
\includegraphics[width=0.4\textwidth]{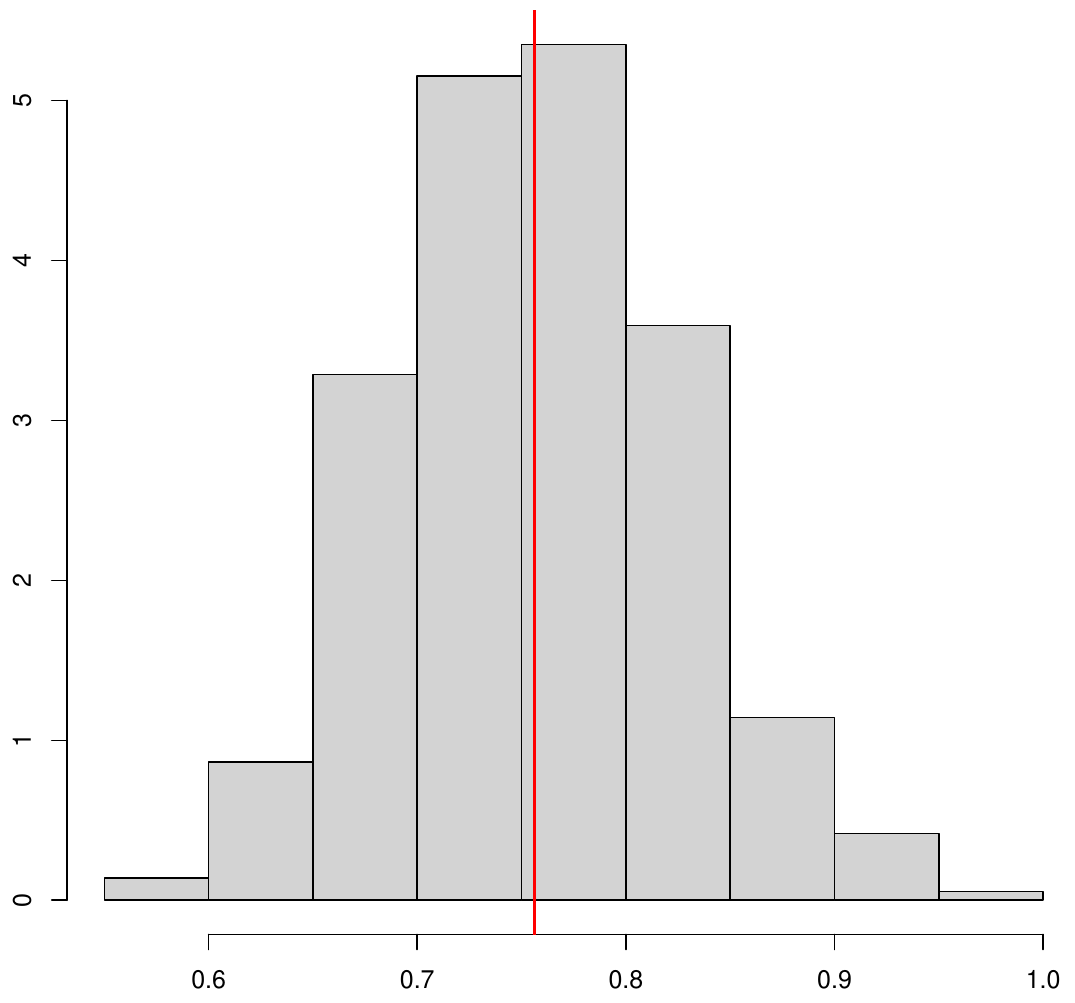}
\end{center}
\caption{\label{fig:flory-sims} \emph{Left:} Plot of $\log \| W_n \|$
against $\log n$ for a single simulated trajectory of the $(W_n,G_n)$ model up to $n=10^5$ steps.
Indicated is also the least-squares linear regression line, which has slope $\approx 0.76$.
\emph{Right:} Histogram of the linear regression slope generated by~$10^3$ samples of
trajectories of length $10^5$ steps; sample mean value $\approx 0.756$ is indicated.
}
\end{figure}

In the next section we describe the model that will be the main focus of the paper,
taking inspiration from the above discussion, and present our main results.
In Section~\ref{sec:barycentre} we will return briefly to the self-interacting random walk described above, and explain in more detail the connection 
to the model developed in the bulk of the paper.

\section{Model and main results}
\label{sec:model}

Let $(\Omega, \cF, \Pr)$ be a probability space
supporting a discrete-time stochastic 
process $Z:= (Z_n, n \in \ZP)$ taking values in $\RP \times \R$ and 
adapted to a filtration
$(\cF_n, n\in\ZP)$;
here and throughout the paper, we denote~$\ZP:=\{0,1,2,\ldots\} = \{ 0 \} \cup  \N$.
Represent $Z_n$, $n \in \ZP$, in Cartesian coordinates by $Z_n = (X_n, Y_n)$,
so that $X_n \in \RP$ and $Y_n  \in \R$. 

For parameters  $\alpha \in (-1,\infty)$, $\beta, \gamma\in \RP$ and $\rho \in (0,\infty)$,
for every $n \in \ZP$
define $\kappa_n : \RP \times \R   \to \RP$ by
\begin{equation}
    \label{eq:kappa-def}
    \kappa_n (z) := \kappa_n ( x, y) := 
        \displaystyle\frac{ \rho |y|^\gamma}{(1+x)^\alpha (1+n)^\beta} ,
        \text{ for all } z := (x,y) \in \RP \times \R.
    \end{equation}
    
We will suppose the process $Z$ is such that, given $\cF_n$, the increment
$Z_{n+1} -Z_n$ is a stochastic perturbation of $\kappa_n (Z_n)$.
More precisely, define the stochastic innovations $\xi = ( \xi_n, n \in \N)$,
with coordinates $\xi_n = (\xi_n^{(1)} , \xi_n^{(2)} ) \in \R^2$, 
through the equalities,
\begin{align}
    \label{eq:X-increment}
    X_{n+1} - X_n & = \kappa_n (Z_n) + \xi_{n+1}^{(1)} , \\
    \label{eq:Y-increment}
    Y_{n+1} - Y_n & = \xi_{n+1}^{(2)} , \text{ for all } n \in \ZP,
\end{align}
so that $\xi_n$ is $\cF_n$-measurable. Then we will
characterize the dynamics of $Z$ through~\eqref{eq:X-increment}--\eqref{eq:Y-increment}
and declaring that $\xi$ satisfies,
for constants $B \in \RP$ and $\delta \in (0,1)$,
the following assumptions.
\begin{description}
\item
[\namedlabel{ass:bounded-jumps}{B}]
\emph{Bounded innovations.}
For all $n \in \ZP$,
\begin{align}
\label{eq:xi-jump-bound}
\Pr ( \| \xi_{n+1} \| \leq B  ) & = 1 . \end{align}
\item
[\namedlabel{ass:martingale}{M}]
\emph{(Sub)martingale innovations.}
For all $n \in \ZP$,
\begin{equation}
  \label{eq:xi-martingale}
0 \leq \Exp ( \xi^{(1)}_{n+1} \mid \cF_n )  \leq B \1 { X_n \leq B }, ~\text{and}~
\Exp ( \xi^{(2)}_{n+1} \mid \cF_n ) = 0.
\end{equation}
\item 
[\namedlabel{ass:ellipticity}{E}]
\emph{Uniform ellipticity of vertical innovations.}
For all $n \in \ZP$, 
\begin{align}
\label{eq:xi-ellipticity}
\Pr (| \xi^{(2)}_{n+1}| \geq \delta \mid \cF_n )  \geq \delta . 
\end{align}
\end{description}
We suppose that the initial state $Z_0 = (X_0, Y_0) \in \RP \times \R$
is arbitrary (but fixed). 
A  consequence of~\eqref{eq:X-increment}--\eqref{eq:xi-jump-bound} is
\begin{equation}
    \label{eq:XY-jump-bound}
    \Pr ( | X_{n+1} - X_n - \kappa_n (Z_n) | \leq B ) = \Pr ( | Y_{n+1} - Y_n | \leq B ) = 1.
\end{equation}
From~\eqref{eq:X-increment}, \eqref{eq:Y-increment}, and~\eqref{eq:xi-martingale} it follows that
 \begin{align}
\label{eq:X-drift-bound} 
0 \leq \kappa_n (Z_n ) \leq \Exp ( X_{n+1} - X_n \mid \cF_n ) & \leq \kappa_n (Z_n) + B \1 {X_n \leq B }  ; \\
\label{eq:Y-drift} 
\Exp ( Y_{n+1} - Y_n \mid \cF_n ) & = 0;
\end{align}
in particular, $Y  = (Y_n, n \in \ZP)$ is a martingale and $X  = (X_n, n \in \ZP)$ is a submartingale.
Moreover, from~\eqref{eq:Y-increment} and~\eqref{eq:xi-ellipticity} we have
\begin{equation}
\label{eq:Y-ellipticity}
 \Pr (| Y_{n+1} -Y_n| \geq \delta \mid \cF_n )  \geq \delta .
 \end{equation}
Note that we do not necessarily suppose that $Z$ is a Markov process,
although in many natural examples that will be the case (and, for $\beta \neq 0$,
the Markov process will be inhomogeneous in time). Note also that
we make no independence assumption on the components of~$\xi$.
We give some examples;
the first, Example~\ref{ex:1},
gives a process  that lives on $\ZP \times \Z$;
our assumptions~\eqref{ass:bounded-jumps}, \eqref{ass:martingale}, and~\eqref{ass:ellipticity}
are broad enough that a whole host of similar examples, on $\ZP \times \Z$ or on $\RP \times \R$,
 can be constructed.

 \begin{example}
 \label{ex:1}
Take $Z_0 \in \ZP \times \Z$. For $x \in \R$, recall that $\lceil x \rceil$ is the
unique
integer for which $\lceil x \rceil -1 < x \leq \lceil x \rceil$.
Define the $\cF_n$-measurable random variable $\varphi_n :=  \lceil  \kappa_n (Z_n) \rceil -  \kappa_n (Z_n) \in [0,1)$.
Suppose that 
\begin{align*}
\Pr ( \xi_{n+1}^{(1)} = \varphi_n, \, \xi_{n+1}^{(2)} = 0 \mid \cF_n ) & = \frac{1 -\varphi_n}{2} + \frac{\varphi_n}{2} \1 { X_n = 0},  \\
\Pr ( \xi_{n+1}^{(1)} = \varphi_n -1, \, \xi_{n+1}^{(2)} = 0 \mid \cF_n ) & = \frac{\varphi_n}{2} \1 { X_n \geq 1} ,  \\
\Pr (  \xi_{n+1}^{(1)} =0, \, \xi_{n+1}^{(2)} = +1 \mid \cF_n ) & = \Pr ( \xi_{n+1}^{(1)} =0, \,  \xi_{n+1}^{(2)} = -1 \mid \cF_n ) = \frac{1}{4} ,
\end{align*}
for all $n \in \ZP$, these being the only possible transitions. 
Then~\eqref{eq:X-increment}
shows that $Z_n \in \ZP \times \Z$ for all $n \in \ZP$,
and 
the assumptions~\eqref{ass:bounded-jumps}, \eqref{ass:martingale}, and~\eqref{ass:ellipticity} 
are satisfied, with constants $B=1$ and $\delta = 1/2$.
\end{example}

\begin{example}
\label{ex:mv1}
    For $\gamma =0$, 
    our model is similar to that studied by
    Menshikov \& Volkov~\cite{mv2}.
    Menshikov \& Volkov considered time-inhomogeneous Markov chains 
    on $\RP$, satisfying certain
   regularity conditions, including bounded increments,
    for which
    \begin{equation}
    \label{eq:mv}
    \Exp ( X_{n+1} - X_n \mid \cF_n ) \geq \frac{\rho}{(1+X_n)^\alpha (1+n)^\beta} , 
    \end{equation}
    where $\rho >0$, $\beta \geq 0$, and $\alpha + \beta > 0$ (we have reparametrized the
formulation from that of~\cite{mv2} for compatibility with ours).
Theorem~1 of~\cite{mv2} established transience, i.e., $X_n \to \infty$, a.s.,
under the conditions 
\begin{equation}
    \label{eq:mv-regime}
    0 \leq \beta < 1, \text{ and } -\beta < \alpha < 1-2\beta.
\end{equation}
Subsequently, Gradinaru \& Offret~\cite{go}
made a thorough investigation of a time-inhomogeneous diffusion
on $\RP$ analogous to the Markov chain of~\cite{mv2}, proving sharp results on
asymptotic behaviour, using methods specific to the continuum setting. 
A simple consequence of our main theorem below is a strengthening of
Theorem~1 of~\cite{mv2},
quantifying the superdiffusive transience of~$X_n$: see Example~\ref{ex:mv2}
below.
\end{example}

Define the characteristic exponent
\begin{equation}
    \label{eq:chi-def}
    \chi := \chi (\alpha, \beta, \gamma ) := \frac{2 + \gamma-2\beta}{2+2\alpha} .
\end{equation}

Write  $\alpha^+ := \max (0,\alpha)$. The following is our main result.

\begin{theorem}
\label{thm:main-thm}
Suppose that~\eqref{ass:bounded-jumps}, \eqref{ass:martingale}, and~\eqref{ass:ellipticity} hold.
Suppose that $\alpha \in (-1,\infty)$ and $\gamma, \beta \in \RP$ are such that
\begin{equation}
    \label{eq:3-parameters-condition-superdiffusive}
   1+ \gamma > \alpha^+ + 2 \beta .
\end{equation}
    Then, for $\chi = \chi  (\alpha, \beta, \gamma )$ defined at~\eqref{eq:chi-def}, 
  \begin{equation}
    \label{eq:log-limit}
     \lim_{n\to \infty} \frac{\log X_n}{\log n} = \chi  , \as \end{equation}
\end{theorem}

\begin{remarks}
\phantomsection
\label{rems:main-thm}
\begin{myenumi}[label=(\roman*)]
\item\label{rems:main-thm-i} 
    Note that~\eqref{eq:3-parameters-condition-superdiffusive}
    implies that $2 + \gamma - 2\beta > 1 +\alpha^+ \geq 1 +\alpha$,
    so that $\chi$ given by~\eqref{eq:chi-def}
    satisfies~$\chi > 1/2$. Hence Theorem~\ref{thm:main-thm} shows that
    $X$ exhibits \emph{super-diffusive transience}.    It is  possible that
    $\chi > 1$, in which case the transience is even \emph{super-ballistic},
    since the dynamics~\eqref{eq:X-increment} can permit increments of arbitrary size for the process.
\item\label{rems:main-thm-ii} 
Our main interest in the present paper is identifying the scaling exponent~$\chi$.
It is natural also to investigate conditions under which there
holds weak convergence~of $n^{-\chi} X_n$.
Such convergence would seem to require, at least, weak convergence
of the martingale path-sums studied in Appendix~\ref{sec:one-dimensional-bounds},
which will need further hypotheses on second-moments of the increments of~$Y$. For example,
a natural additional hypothesis is
\begin{equation}
    \label{eq:Y-variance}
  \lim_{n \to \infty}  \Exp ( (Y_{n+1} -Y_n)^2 \mid \cF_n ) = \sigma^2, \text{ in probability},
\end{equation}
which enables one to apply the martingale invariance principle~\cite{mcleish} to show 
convergence of $n^{-1/2} Y_{\lfloor nt \rfloor}$ to Brownian motion, and in this case the candidate
limit for $n^{-\chi} X_n$ would involve an integral functional of Brownian motion. 
Since Example~\ref{ex:1},
and many other natural examples, satisfy~\eqref{eq:Y-variance}, we would expect
 $n^{-\chi} X_n$ to possess a distributional limit in many cases.
 Precise conditions under which an analogous convergence result is established for a diffusion version of the model,
 going beyond the setting of martingale~$Y$, are described in~\cite{bcmw}.
\end{myenumi}
 \end{remarks}

\begin{example}
\label{ex:34}
The self-interacting random walk model that inspired this work
(see Section~\ref{sec:intro} for a description and Section~\ref{sec:barycentre} for more mathematical details)
motivates the parameter choice $\alpha = \gamma =1$, $\beta =0$, so that~\eqref{eq:X-drift-bound} gives
\[ \Exp ( X_{n+1} - X_n \mid \cF_n ) = \frac{ \rho |Y_n|}{1+X_n}, \text{ on } \{ X_n > B \}. \]
In this case,  Theorem~\ref{thm:main-thm}
shows that $\log X_n / \log n \to 3/4$, a.s. The appearance of the famous Flory exponent~$3/4$~\cite{flory,flory2}
was one of the features of this model that captured our attention.
\end{example}

\begin{example}
\label{ex:mv2}
    For $\gamma =0$, 
the conclusion of Theorem~\ref{thm:main-thm} is that, a.s.,
\[ \lim_{n \to \infty} \frac{\log X_n}{\log n} = \frac{1- \beta}{1+\alpha} > \frac{1}{2} ,\]
which establishes a superdiffusive rate of escape for the class of processes studied in~\cite{mv2}
and described in Example~\ref{ex:mv1} above. Note that the condition~\eqref{eq:mv-regime}
from~\cite{mv2} translates to $\frac{1}{2} < \frac{1-\beta}{1+\alpha} < 1$;
unlike our model, the process in~\cite{mv2} is always \emph{sub-ballistic}, since it has bounded increments.  

Gradinaru \& Offret~\cite{go} studied the time-inhomogeneous diffusion 
\begin{equation}
   \label{eq:go}
   \ud X_t = \rho t^{-\beta} X_t^{-\alpha} \ud t + \ud W_t, \end{equation}
where~$W$ is Brownian motion and, again,
we have reparametrized to match our formulation. 
Theorem~4.10(i) of~\cite[p.~200]{go} says that, if $\rho >0$, and $2\beta < 1-\alpha <2$, then
\[ \lim_{t \to \infty} t^{-\frac{1-\beta}{1+\alpha}} X_t 
= c_{\alpha, \beta}, \as, \]
for an explicit $c_{\alpha, \beta} \in (0,\infty)$.
\end{example}

The focus of this paper is the case where $\chi > 1/2$.
For $\chi \leq 1/2$, one no longer expects
superdiffusivity, and it becomes
an interesting question to investigate whether $X$, or indeed $Z$,
is recurrent or not, i.e., is it true that $\liminf_{n \to \infty} X_n < \infty$ a.s., or $\liminf_{n \to \infty} \| Z_n\| < \infty$, a.s. We pose the following.

\begin{problem}
Suppose that $\alpha \in (-1,\infty)$, $\gamma, \beta \in \RP$,
but $1 + \gamma \leq \alpha + 2 \beta$. Classify when $X$ and $Z$ are recurrent
or diffusively transient.
\end{problem}
\begin{remark}
    In the case $\gamma =0$, it was shown in Theorem~2 of~\cite{mv2}
    that a closely related model is recurrent when $\beta \geq 0$
    and $\alpha > \max (-\beta, 1 -2\beta)$; cf.~Remark~\ref{ex:mv1}.
    In the boundary cases where~(i) $\alpha=1$ and $\beta=0$~\cite{lamp}, or (ii)~$\beta = 1 = - \alpha$~\cite{mv2},
    the recurrence/transience classification depends on $\rho$, showing
    that, in general, the
    critical-parameter case is delicate.
\end{remark}

The paper is organized as follows. In Section~\ref{sec:Y-process}
we establish  a.s.~asymptotics for the vertical process~$Y$
and associated additive functionals, and moments bounds; these
quantify the fact that $Y$ is diffusive.
The a.s.~asymptotics in Section~\ref{sec:Y-process} are deduced from
 asymptotics for additive functionals of
one-dimensional martingales that we defer to Appendix~\ref{sec:one-dimensional-bounds}.
In Section~\ref{sec:X-asymptotics} we study the (more difficult) behaviour of the process~$X$,
and prove Theorem~\ref{thm:main-thm}.
Then in Section~\ref{sec:barycentre} we return to the self-interacting random walk
model described  in Section~\ref{sec:intro} to explain, heuristically,
the relationship between that model and the phenomena exhibited in Theorem~\ref{thm:main-thm},
and to formulate a conjecture about its behaviour.

In what follows, we will denote by $C$ various positive constants which may differ
on each appearance, and whose exact values are immaterial for the  proofs.

\section{Asymptotics for the vertical process}
\label{sec:Y-process}

The following proposition 
gives  the long-term behaviour of 
the component $Y$ of the process $Z=(X, Y)$.
    The proof
relies on Theorem~\ref{thm:S-bounds}
in~\S\ref{sec:one-dimensional-bounds},
which gives a slightly more general result for processes on $\RP$.

\begin{proposition}
\label{prop:|Y|-sum}
Suppose that the process~$Y$ satisfies~\eqref{eq:XY-jump-bound}, \eqref{eq:Y-drift}, and~\eqref{eq:Y-ellipticity}.
Then 
\begin{equation}
\label{eq:|Y|-limit-diffusion} \lim_{n \to \infty} \frac{ \log \max_{0\leq m \leq n} |Y_m|}{\log n } =   \frac{1}{2} , \as \end{equation}
Moreover, for any $\beta \in \R$ and $\gamma \in \RP$ with $2\beta \leq 2 +\gamma$,
\begin{equation}
\label{eq:|Y|-limit} \lim_{n \to \infty} \frac{ \log \sum_{m=1}^n m^{-\beta} |Y_m|^\gamma}{\log n } = 1 + \frac{\gamma}{2} -\beta, \as \end{equation}
\end{proposition}

Before giving the proof of Proposition~\ref{prop:|Y|-sum}, we need one elementary technical result
related to the ellipticity hypothesis~\eqref{eq:Y-ellipticity}.
For $x \in \R$, write $x^+ = x \1 { x >0 }$
and $x^- = - x \1 { x < 0 }$, so that $x = x^+ -x^-$ and $|x| = x^+ + x^-$.

\begin{lemma}
\label{lem:ellip-technical}
Let $\zeta$ be a random variable on a probability space $(\Omega,\cF,\Pr)$ and let $\cG$ be a sub-$\sigma$-algebra of $\cF$.
Suppose that there exist $B \in \RP$ and $\delta >0$ such that $\Pr ( | \zeta | \leq B) = 1$,
$\Exp ( \zeta \mid \cG ) = 0$, a.s., and $\Pr ( | \zeta | \geq \delta \mid \cG ) \geq \delta$, a.s. Then, a.s., 
\[ \Pr \Bigl(\zeta \geq \frac{\delta^2}{4} \Bigmid\cG \Bigr) \geq \frac{\delta^2}{4B}, \text{ and }\Pr \Bigl(\zeta \leq -\frac{\delta^2}{4} \Bigmid\cG \Bigr) \geq \frac{\delta^2}{4B} .\] 
\end{lemma}
\begin{proof}
Since $\Exp ( \zeta \mid \cG ) = 0$, we have
$\Exp ( \zeta^+ \mid \cG) = \Exp ( \zeta^- \mid \cG)$, and hence
\[ 2 \Exp ( \zeta^+ \mid \cG) = \Exp ( | \zeta | \mid \cG ) \geq 
\Exp ( | \zeta | \1{ | \zeta | \geq \delta } \mid \cG ) \geq 
\delta^2, \as, \]
by the fact that $\Pr ( | \zeta | \geq \delta \mid \cG ) \geq \delta$. 
Then, since $\zeta^+ \leq | \zeta | \leq B$, a.s.,
\[ \frac{\delta^2}{2} \leq \Exp  ( \zeta^+ \mid \cG) \leq \frac{\delta^2}{4} + B \Pr \Bigl( \zeta^+ \geq \frac{\delta^2}{4} \Bigmid\cG \Bigr) .\]
This gives the first statement in the lemma, and the same argument applies to $\zeta^-$.
\end{proof}

For  $x = (x_0, x_1, \ldots ) \in \R^{\ZP}$ a real-valued
sequence, we denote the $n$th difference by $\Delta x_n := x_{n+1} - x_n$.

\begin{proof}[Proof of Proposition~\ref{prop:|Y|-sum}]
For $n \in \ZP$, set $S_n = |Y_n| \in \RP$. We verify that $S =(S_n, n \in \ZP)$
satisfies the hypotheses of 
Theorem~\ref{thm:S-bounds}. 
First, the uniform bound on increments~\eqref{eq:S-jump-bound} follows from the bound
$|Y_{n+1}| - |Y_n| \leq |\Delta Y_{n}|$ (triangle inequality) together with 
the bound on $| \Delta Y_{n}|$ from~\eqref{eq:XY-jump-bound}. 
The martingale property~\eqref{eq:Y-drift} implies that $S$ is a submartingale; i.e., $\Exp ( \Delta S_n \mid \cF_n ) \geq 0$.
Moreover,
it also follows from~\eqref{eq:XY-jump-bound} that
if $| Y_n | \geq B$ then
$\Delta | Y_{n} | = \sign (Y_n) \cdot \Delta Y_n$, where $\sign (x) := \1 {x > 0 } - \1 { x < 0}$ 
is the sign of $x \in \R$. Hence,
\[    \Exp (\Delta S_n  \mid \cF_n )   = \sign (Y_n) \Exp ( \Delta Y_{n} \mid \cF_n ) = 0 , \text{ on } \{ S_n \geq B \},\]
by~\eqref{eq:Y-drift}; thus~\eqref{eq:S-zero-drift} holds (with $x_0 = B$). Finally,
it follows from Lemma~\ref{lem:ellip-technical} with~\eqref{eq:XY-jump-bound}, \eqref{eq:Y-drift},
and~\eqref{eq:Y-ellipticity} that there exists $\delta'>0$ (depending on $B$ and $\delta$) such that
\[ \Pr ( \Delta Y_{n} \geq \delta' \mid \cF_n) \geq \delta', \as, \text{ and } \Pr ( \Delta Y_{n}  \leq - \delta' \mid \cF_n) \geq \delta', \as \]
Hence,
\begin{align*} \Pr (  \Delta S_n  \geq \delta' \mid \cF_n )
& \geq \Pr ( \Delta Y_n \geq \delta' \mid \cF_n) \1 { Y_n \geq 0 } + \Pr ( \Delta Y_{n}  \leq -\delta' \mid \cF_n)\1 { Y_n < 0 }  , \end{align*}
which, a.s., is at least $\delta'$.
Thus~\eqref{eq:S-ellipticity} also holds (with $\delta = \delta'$).
Hence we may apply Theorem~\ref{thm:S-bounds} to get the result.
\end{proof}

The next result gives moments bounds that quantify the fact that $Y$ is diffusive.
 
\begin{lemma}
\label{lem:Y-moments}
Suppose that the process~$Y$ satisfies~\eqref{eq:XY-jump-bound}, \eqref{eq:Y-drift}, and~\eqref{eq:Y-ellipticity}.  For every $\gamma\in \RP$ there exist constants
$a_\gamma$ (depending on $\gamma$, $\delta$, and $B$) and $A_\gamma$ (depending on $\gamma$, $B$, and $\Exp ( |Y_0|^\gamma )$), such that $0 < a_\gamma  \leq A_\gamma < \infty$ and 
   \begin{equation}
   \label{eq:Y-moments-bounds}
    a_\gamma n^{\gamma/2} \leq \Exp(|Y_n|^{\gamma}) \leq A_\gamma  n^{\gamma/2}, 
\text{ for all } n\in \N.
    \end{equation} 
\end{lemma}
\begin{proof}
First, we prove by induction on $k\in \N$ that
there exist constants $a_k >0$ (depending on $B$, $k$, and $\delta$)
and $A_k < \infty$ (depending on $B$, $k$, and $\Exp ( Y_0^{2k-2})$),
such that,  for all $k \in \N$,
\begin{equation}
    \label{eq:Y-moments-induction-hypothesis}
    a_k n^{k} \leq 
 \Exp ( Y_n^{2k} ) - \Exp (Y_0^{2k} ) \leq A_k n^{k}, \text{ for all } n \in \N.
 \end{equation}
For $k \in \N$,
Taylor's theorem with Lagrange
remainder says that, 
\begin{equation}
    \label{eq:taylor-2k}
 (x+u)^{2k} - x^{2k}
= 2k x^{2k-1} u + (2k-1) k u^2 (x+ \theta u)^{2k-2}, \text{ for all } x \in \R, u \in \R,
\end{equation}
where $\theta := \theta (x,u) \in (0,1)$. 
The right-most term in~\eqref{eq:taylor-2k} is always non-negative, so 
\begin{equation}
    \label{eq:taylor-lower}
 (x+u)^{2k} - x^{2k}
\geq  2k x^{2k-1} u +  (2k -1) k u^2 x^{2k-2} \1 { u x > 0  }, \text{ for } x, u \in \R.
\end{equation}
In particular,  since $\Exp ( \Delta Y_n \mid \cF_n ) = 0$,
it follows from taking (conditional) expectations in~\eqref{eq:taylor-lower} applied with $x = Y_{n}$ and $u = \Delta Y_n$ that
\begin{align*} 
 \Exp ( Y_{n+1}^{2k} - Y_n^{2k} \mid \cF_n ) & \geq Y_n^{2k-2} \Exp \bigl[ ( \Delta Y_n )^2 \1 {  \Delta Y_n  > 0 } \bigmid \cF_n \bigr] \1{ Y_n \geq 0} \\
 & {}  \qquad {} + Y_n^{2k-2} \Exp \bigl[ ( \Delta Y_n )^2 \1 {  \Delta Y_n  < 0 } \bigmid \cF_n \bigr] \1{ Y_n \leq 0} , \text{ for all } k \in \N.
\end{align*}
By Lemma~\ref{lem:ellip-technical} with~\eqref{eq:XY-jump-bound}, \eqref{eq:Y-drift}, and~\eqref{eq:Y-ellipticity}, 
\[  \Exp \bigl[ ( \Delta Y_n )^2 \1 { \Delta Y_n > 0 } \bigmid \cF_n \bigr] 
\geq (\delta^2/4)^2 \Pr ( { \Delta Y_n \geq \delta^2/4 } \mid \cF_n )
\geq \frac{\delta^6}{64 B}, \as,
\]
and the same bound holds with $\{ \Delta Y_n < 0 \}$ instead in the indicator.
Thus there exists $a >0$ (depending on $\delta$ and $B$) such that, for all $k \in \N$,
\[ \Exp ( Y_{n+1}^{2k} - Y_n^{2k}  )
\geq a \Exp ( Y_n^{2k-2} ), \text{ for all } n \in \ZP. \]
Hence
\begin{equation}
    \label{eq:induction-lower-bound}
\Exp ( Y_{n}^{2k} ) \geq \Exp ( Y_{n}^{2k} ) - \Exp ( Y_0^{2k}  )
= \sum_{m=0}^{n-1} \Exp ( Y_{m+1}^{2k} - Y_m^{2k}  )
\geq a \sum_{m=0}^{n-1} \Exp ( Y_m^{2k-2} ) .\end{equation}
Induction on $k \in \N$ using~\eqref{eq:induction-lower-bound}
establishes the lower bound in~\eqref{eq:Y-moments-induction-hypothesis}.

For the upper bound, it follows from~\eqref{eq:taylor-2k} that, for all 
$| x | \geq B$ and all $|u| \leq B$,
\begin{align*} 
 (x+u)^{2k} - x^{2k}
& \leq 2k x^{2k-1} u + (2k-1) k u^2 (|x|+ B)^{2k-2} \\
& \leq 2k x^{2k-1} u + 2^{2k-2} B^2 (2k-1) k x^{2k-2}.
\end{align*}
 On the other hand, for $| x| \leq B$ and $|u| \leq B$, we have simply that
$(x+u)^{2k} \leq (2B)^{2k}$. Thus there is a constant $A'_k < \infty$
(depending only on $k$ and $B$) such that
\[ (x+u)^{2k} - x^{2k}
\leq 2k x^{2k-1} u  + A'_k (1 + x^{2k-2} ) , \text{ for all } x \in \R, \, |u| \leq B.\]
In particular,  since $\Exp (\Delta Y_n \mid \cF_n ) = 0$, 
taking $x= Y_n$ and $u = \Delta Y_n$, we obtain
\[
\Exp ( Y_{n+1}^{2k} - Y_n^{2k} \mid \cF_n )
\leq A'_k ( 1 + Y_n^{2k-2} ) , \as, \]
and hence
\begin{equation}
    \label{eq:induction-upper-bound}
\Exp ( Y_{n}^{2k} ) - \Exp ( Y_0^{2k}  )
= \sum_{m=0}^{n-1} \Exp ( Y_{m+1}^{2k} - Y_m^{2k}  )
\leq A'_k n + A'_k \sum_{m=0}^{n-1} \Exp ( Y_m^{2k-2} ) .\end{equation}
Induction on $k \in \N$ using~\eqref{eq:induction-upper-bound}
establishes the upper bound in~\eqref{eq:Y-moments-induction-hypothesis}.

Given~\eqref{eq:Y-moments-induction-hypothesis}, suppose that $2k-2 < \gamma \leq 2k$, $k \in \N$.
Then Lyapunov's inequality,
$\Exp ( | Y_n|^\gamma )  \leq ( \Exp [ |Y_n|^{2k} ] )^{\gamma/(2k)}$,
gives the upper bound in~\eqref{eq:Y-moments-bounds},
since $\Exp ( | Y_0|^{2k-2} )$ (on which the constants depend)
can be bounded in terms of $\Exp ( |Y_0|^\gamma )$, as claimed. In the other direction, first note that the upper bound in~\eqref{eq:Y-moments-induction-hypothesis} shows that, for every $k \in \N$,
$\sup_{n \in \N} \Exp [ ( Y_n/\sqrt{n} )^{2k} ]< \infty$,
which implies that, for every $k \in \N$, the collection
$(Y_n/\sqrt{n})^{2k}$, $n \in \N$, is uniformly integrable. Hence there exists $b_k \in \N$ such that
\[ \sup_{n \in \N} \Exp [ (Y_n/\sqrt{n} )^{2k} \1 { Y_n > b_k \sqrt{n} } ] \leq a_k /2. \]
Hence, by the lower bound
in~\eqref{eq:Y-moments-induction-hypothesis},
\[ \Exp [ (Y_n/\sqrt{n} )^{2k} \1 { Y_n \leq b_k \sqrt{n} } ] \geq 
a_k - \Exp [ (Y_n/\sqrt{n} )^{2k} \1 { Y_n > b_k \sqrt{n} } ]
\geq a_k/2, \]
for all $n \in \N$.
Now, since $\gamma \leq 2k$,
\[ \Exp [ |Y_n|^\gamma ]
 \geq \Exp [ |Y_n|^{2k} \cdot |Y_n|^{\gamma-2k} \1 { Y_n \leq b_k \sqrt{n} }  ]
\geq b_k^{\gamma-{2k}} n^{(\gamma/2)-k} \Exp[ |Y_n|^{2k} ],
\]
which, together with the fact that
$b_k^{\gamma-2k} \geq b_k^{-2}$,
gives the lower bound in~\eqref{eq:Y-moments-bounds}.
\end{proof}

\section{Asymptotics for the horizontal process}
\label{sec:X-asymptotics} 

\subsection{Confinement}
\label{sec:confinement}

The following `confinement' result is a crucial ingredient in our analysis; roughly speaking, it gives a (deterministic) upper bound on the relative sizes of $Y_n$ and $X_n$, and will thus permit us to use a Taylor expansion to study $f(X_{n+1} ) - f(X_n)$ using~\eqref{eq:X-increment}. The reason for confinement is a sort of negative feedback in the dynamics: if $Y_n$ is very large compared to $X_n$, then the drift in $X_n$ through~\eqref{eq:X-increment} is large, which will mean $X_n$ tends to increase much faster than $Y_n$.

\begin{proposition}
\label{prop:confinement}
Let $\alpha \in (-1,\infty)$ and $\beta, \gamma \in \RP$. 
Suppose that~\eqref{ass:bounded-jumps} holds.
Then there exists a constant $C_0 := C_0 (B, \rho, \alpha, \gamma ) < \infty$ such that
\begin{equation}
\label{eq:line-confinement}
\sup_{n \in \ZP}\frac{\kappa_n (Z_n)}{(1+X_n)} \leq \rho \max ( |Y_0|^\gamma, C_0), \as
\end{equation}
\end{proposition}

Before giving the proof of the last result, we state one important consequence.

\begin{corollary}
    \label{cor:confinement-bound} 
Suppose that~\eqref{ass:bounded-jumps} holds. 
Then there is a constant $C < \infty$
(depending on $B, \rho, \alpha, \gamma$ as well as $|Y_0|$) such that,
for all $n \in \ZP$,
\[ \Delta X_n \leq C \max (1 , X_n), \text{ and } X_{n+1}  \leq C \max (1 , X_n). \]
\end{corollary}
\begin{proof}
    From~\eqref{eq:XY-jump-bound}, we have $\Delta X_n \leq B + \kappa_n (Z_n)$,
    which together with~\eqref{eq:line-confinement} yields
    $\Delta X_n \leq B + \rho \max ( |Y_0|^\gamma, C_0) (1+X_n)$. The claimed result follows.
\end{proof}

\begin{proof}[Proof of Proposition~\ref{prop:confinement}]
Let $\alpha \in (-1,\infty)$ and $\beta, \gamma \in \RP$. 
Define, for $n \in \ZP$,
\begin{equation}
\label{eq:zeta-def}
  \zeta_n := \frac{\max (B, |Y_n|)^{\gamma}}{(1+n)^\beta (1+X_n)^{1 + \alpha}} \geq \rho^{-1} \frac{\kappa_n(Z_n)}{1+X_n},
\end{equation}
by~\eqref{eq:kappa-def}.
To verify~\eqref{eq:line-confinement}, it suffices to prove that $\sup_{n \in \ZP} \zeta_n \leq \max ( |Y_0|^\gamma, C_0)$, a.s.
By the bound on $|\Delta Y_{n} |$ from~\eqref{eq:XY-jump-bound}, 
if $|Y_n| \geq B$, then $\max ( B, |Y_{n+1}|) \leq |Y_n| + B \leq 2 |Y_n| = 2 \max ( B, |Y_n|)$,
while if $|Y_n| \leq B$, then $\max (B, |Y_{n+1} |) \leq 2 B = 2 \max (B, |Y_n|)$. Hence
\begin{equation}
\label{eq:max-B-Y}
\max ( B, |Y_{n+1}|)  \leq 2 \max ( B, |Y_{n}|) , \as 
\end{equation}
To proceed we consider two cases. We first consider the case
\begin{equation}
\label{eq:large_case}    
\frac{ {\rho} |Y_n|^\gamma}{(1 + n )^\beta(1 + X_n)^{1 + \alpha}} \geq 2^{\gamma/(1+\alpha)}+B,
\end{equation}
and then we consider  
\begin{equation}
\label{eq:small_case}    
\frac{ {\rho} |Y_n|^\gamma}{(1 + n )^\beta(1 + X_n)^{1 + \alpha}}  < 2^{\gamma/(1+\alpha)}+B.
\end{equation}
First, if~\eqref{eq:large_case} holds 
then, by~\eqref{eq:X-increment},
\begin{align*}
1 + X_{n+1}  & \geq  X_n+ \frac{\rho |Y_n|^\gamma}{(1+n)^\beta (1+X_n)^\alpha} + \xi^{(1)}_{n+1} \\
& \geq (2^{\gamma/(1+\alpha)}+B) (1+ X_n) - B
\geq 2^{\gamma/(1+\alpha)} (1+ X_n),
\end{align*}
using~\eqref{eq:xi-jump-bound} and the fact that $X_n \geq 0$.
Hence, using  the definition~\eqref{eq:zeta-def} for $\zeta_{n+1}$ and~\eqref{eq:max-B-Y},
\begin{align}
\label{eq:decrease-2g}
\zeta_{n+1} = \frac{\max ( B, |Y_{n+1}|)^\gamma}{(2+ n)^\beta (1+X_{n+1})^{1 + \alpha}} 
& \leq  
\frac{2^\gamma \max ( B, |Y_{n}|)^\gamma}{(1+n)^\beta (2^{\gamma/(1+\alpha)}(1+X_n))^{1 + \alpha}} \nonumber\\
& =  \frac{\max (B,|Y_n|)^\gamma}{(1+n)^\beta (1+X_n)^{1 + \alpha}}
=  \zeta_n.
\end{align}
For the second case, i.e.~when~\eqref{eq:small_case} holds, note first 
that $X_{n+1} \geq \max (0, X_n - B) $,
so that 
   \begin{equation}
\label{eq:bound-1g}
1 + B + X_{n+1} \geq \max (1+B, X_{n}) \geq (1 + X_n)/2.
\end{equation}
Moreover, using~\eqref{eq:max-B-Y} 
and the fact that $(1 + B)(1 + X_{n+1}) \geq 1 + B + X_{n+1}$, we obtain 
\begin{align}
 \zeta_{n+1} = \frac{\max(B, |Y_{n+1}|)^{\gamma}}{(2+n)^\beta (1+X_{n+1})^{1 + \alpha}} 
  & \leq  
  \frac{ 2^\gamma (1+B)^{1+\alpha} \max ( B, |Y_n|)^\gamma}{(1+n)^\beta (1+ B + X_{n+1})^{1 + \alpha}} \nonumber\\\label{upper_xi_second_case}
  & \leq   \frac{2^{\gamma + 1 +
    \alpha}  (1+B)^{1+\alpha} \max (B, |Y_n|)^\gamma} {(n+1)^\beta (1+X_n)^{1 + \alpha}} ,
    \end{align}
by~\eqref{eq:bound-1g}.  If~\eqref{eq:small_case} holds, then,
since $\alpha >-1$ and $\beta \geq 0$, from~\eqref{upper_xi_second_case} we conclude that
    \begin{equation}
\label{eq:bound-2g}
\zeta_{n+1}
 \leq 
    2^{\gamma + 1
    + \alpha}  (1+B)^{1+\alpha} \rho^{-1} \max ( 2^{\gamma/(1+\alpha)}+B , \rho B^{\gamma}) =: C'_0.
\end{equation}
The combination of~\eqref{eq:decrease-2g} and~\eqref{eq:bound-2g} implies that
for each $n$ it is  the case that at least one of the statements~(i) $\zeta_{n+1} \leq \zeta_n$, a.s.,
and (ii)~$\zeta_{n+1} \leq C'_0$, a.s., must hold. It follows that, a.s., $\zeta_n \leq \max ( \zeta_0, C'_0)$ for all $n \in \ZP$. Since $\zeta_0 \leq \max ( B , |Y_0|)^\gamma$ we deduce~\eqref{eq:line-confinement}.
 \end{proof}

 \subsection{Bounds on increments and moments}
\label{sec:compensator}

To study the process $X$ we will study the transformed process $X^\nu$ for some $\nu \geq 1$.
First we need upper and lower bounds on the increments $X_{n+1}^\nu - X_n^\nu$, or
the (conditional) expectations thereof. This is the purpose of the next result.

 \begin{lemma}
 \label{lem:x-increments}
It holds that, for every $\nu \geq 1$ and all $n \in \ZP$,
\begin{equation}
    \label{eq:increment-nu-lower}
 X_{n+1}^{\nu} - X_n^{\nu}  
\geq \nu X_n^{\nu-1} \Bigl( \kappa_n (Z_n) + \xi^{(1)}_{n+1} \Bigr), \as,
\end{equation}
Furthermore,
 suppose that~\eqref{ass:bounded-jumps} holds, and that $\alpha \in (-1,\infty)$, and  $\beta, \gamma \in \RP$.
 Then,
for every $\nu \geq 1$, there is a constant $C \in \RP$ such that,
for all $n \in \ZP$,
\begin{equation}
    \label{eq:mean-increment-nu-upper}
 \Exp (X_{n+1}^{\nu} - X_n^{\nu} \mid \cF_n )
\leq C X_n^{\nu -1} \kappa_n (Z_n)    + C  X_n^{\nu-2} , \text{ on } \{ X_n \geq 1+ 2B \},
\end{equation}
and 
\begin{equation}
    \label{eq:mean-increment-nu-upper-small}
 \Exp (X_{n+1}^{\nu} - X_n^{\nu} \mid \cF_n )
\leq C  , \text{ on } \{ X_n \leq 1+ 2B \}.
\end{equation}
 \end{lemma}
 \begin{proof}
 Let $\nu \geq 1$. 
 Taylor's theorem with Lagrange
remainder says that, for all $x \in \RP$ and all $u \in \R$ with  $x+u > 0$, 
\begin{equation}
    \label{eq:nu-taylor}
D_\nu (x, u) := (x+u)^{\nu} - x^{\nu} - \nu x^{\nu-1} u = \frac{\nu(\nu-1)}{2} u^2 (x+ \theta u)^{\nu -2},
\end{equation}
where $\theta := \theta (x,u) \in (0,1)$.
Here $x + \theta u \geq \min (x,x+u) > 0$,
and so
in~\eqref{eq:nu-taylor} it holds that
 $u^2 (x+ \theta u)^{\nu -2} \geq 0$. 
 On the other hand, if $u = -x$, then $D_\nu (x,-x) = x^\nu (\nu - 1)$.
  Hence,
for all $x \in \RP$ and all $u \in \R$ with  $x+u \geq 0$, 
it holds that
\begin{equation}
    \label{eq:nu-taylor-bounds}
D_\nu (x, u) \geq 0 \text{ if } \nu \geq 1 .
\end{equation}
Applying~\eqref{eq:nu-taylor-bounds} with $x= X_n$ and $u = \Delta X_n$,
and using the fact that $X_n + \Delta X_n = X_{n+1} \geq 0$,
 together with~\eqref{eq:X-increment}, 
 we verify~\eqref{eq:increment-nu-lower}.

Now suppose also that~\eqref{ass:bounded-jumps} holds, $\alpha \in (-1,\infty)$,  and $\beta, \gamma \in \RP$,
so that the hypotheses of Corollary~\ref{cor:confinement-bound} are satisfied.
By Corollary~\ref{cor:confinement-bound},
there is a constant $C < \infty$ such that $X_{n+1} = X_n + \Delta X_n \leq C(1+X_n)$, and hence,
 on $\{ X_n \leq 1 + 2 B \}$, $X_{n+1}^\nu \leq C^\nu (2+2B)^\nu$,
 which yields~\eqref{eq:mean-increment-nu-upper-small}.
Corollary~\ref{cor:confinement-bound} also implies that
  there is a constant $C < \infty$ 
  such that $\Delta X_n  \leq C X_n$ on $\{ X_n \geq 1+2B\}$;
  also, by~\eqref{eq:X-increment} and~\eqref{eq:XY-jump-bound}, $\Delta X_n \geq -B$, and hence
  \begin{equation} 
  \label{eq:delta-x-bounds}
  - \frac{X_n}{2}  \leq  \Delta X_n \leq C X_n, \text{ on } \{ X_n \geq 1+2 B \} .\end{equation}
Then from~\eqref{eq:nu-taylor} applied with $x = X_n$, $u = \Delta X_n$, and using
the upper bound in~\eqref{eq:delta-x-bounds} (when $\nu \geq 2$) or the lower
bound in~\eqref{eq:delta-x-bounds} (when $1 \leq \nu \leq 2$)
it follows that
\begin{equation}
    \label{eq:increment-two-side-nu}
 X_{n+1}^{\nu} - X_n^{\nu}
 \leq \nu X_n^{\nu-1}  \Delta X_n + C ( \Delta X_n )^2  X_n^{\nu -2}  , \text{ on } \{ X_n \geq 1+2B\}.
\end{equation}
From~\eqref{eq:X-increment} and the inequality $(a+b)^2 \leq 2 (a^2+ b^2)$ we have
\[ ( \Delta X_n )^2  \leq 2 (\kappa_n (Z_n))^2 + 2 B^2
\leq 2 (1+X_n) \kappa_n (Z_n) \cdot \frac{\kappa_n (Z_n)}{1+X_n} + 2 B^2 
, \as \]
Hence Proposition~\ref{prop:confinement} implies that
  there is a constant $C < \infty$ for which
\begin{equation}
\label{eq:X-sq-upper-bound}
 ( \Delta X_n )^2 \leq C (1+X_n) \kappa_n (Z_n) + C .\end{equation}
 Then from~\eqref{eq:increment-two-side-nu} with the bound~\eqref{eq:X-sq-upper-bound}, 
 using that $1 + X_n \leq 2X_n$ on $\{ X_n \geq 1 \}$, 
 and~\eqref{eq:X-increment},
\[ 
 X_{n+1}^{\nu} - X_n^{\nu}
 \leq C  X_n^{\nu-1} \kappa_n (Z_n) + \nu X_n^{\nu-1} \xi_{n+1}^{(1)}  + C   X_n^{\nu -2}  , \text{ on } \{ X_n \geq 1 +2B \}, 
\]
 from which, on taking (conditional) expectations and using~\eqref{eq:xi-martingale}, yields~\eqref{eq:mean-increment-nu-upper}.
\end{proof}

The next result gives upper bounds on the growth rates of the moments of~$X_n$.

 \begin{proposition}
 \label{prop:nu-moments}
Suppose that $\alpha \in (-1,\infty)$,   $\beta, \gamma \in \RP$, and that~\eqref{ass:bounded-jumps} and~\eqref{ass:martingale} hold.
Suppose that $1 +\gamma  > \alpha + 2 \beta$. Then for every $\nu  \in \RP$ and all $\eps>0$,
    \begin{equation}
    \label{eq:x-nu-moments}
     \Exp ( X_n^{\nu} ) - \Exp ( X_0^{\nu} ) = O ( n^{\nu \chi+\eps} ), \text{ as } n \to \infty.
    \end{equation}
 \end{proposition}

 Proposition~\ref{prop:nu-moments}, which we prove later in this section,
 already allows us to deduce an asymptotic upper bound for $X_n$. Recall the
 definition of~$\chi$ from~\eqref{eq:chi-def}.

\begin{corollary}
\label{cor:X-limsup}
Suppose that $\alpha \in (-1,\infty)$,   $\beta, \gamma \in \RP$,
and that~\eqref{ass:bounded-jumps} and~\eqref{ass:martingale} hold.
Suppose that $1 +\gamma  > \alpha + 2 \beta$. Then 
\begin{equation}
    \label{eq:Xn-limsup}
 \limsup_{n \to \infty} \frac{\log X_n}{\log n} \leq \chi, \as \end{equation}
\end{corollary}
\begin{proof}
Let $\eps>0$, and define $t_k := k^{1/\eps}$ for $k \in \N$. Observe that
\[
\Pr \Bigl(\sup_{0 \leq m \leq t_k} X_m > t_k^{\chi + 2\eps} \Bigr) =  \Pr \Bigl( \sup_{0 \leq m \leq t_k}X^{2}_m > t_k^{2(\chi + 2\eps)}\Bigr).
\]
By hypothesis~\eqref{ass:martingale}, the process $X$ is a non-negative submartingale.
Then, by Doob's submartingale inequality
and the $\nu = 2$ case of~\eqref{eq:x-nu-moments} it follows that 
\begin{align*}
  \Pr \Bigl( \sup_{0 \leq m \leq t_k} X^{2}_m > t_k^{2(\chi + 2\eps)}\Bigr) 
&  \leq \frac{\Exp ( X_{t_k}^{2} )}{t_k^{2(\chi + 2\eps)}}  
\leq C t_k^{-2\eps} = C k^{-2}, \text{ for all } k \in \N.
\end{align*}
Hence we obtain
\[
\sum_{k \in \N} \Pr \Bigl(\sup_{0 \leq m \leq t_k} X_m > t_k^{\chi + 2\eps} \Bigr) \leq C \sum_{k \in \N} k^{-2} < \infty.
\]
It follows from the Borel--Cantelli lemma that there is a (random) $K \in \N$ 
with $\Pr (K < \infty) =1$ such that $\sup_{0 \leq m \leq t_k} X_m \leq t_k^{\chi+2\eps}$
for all $k \geq K$.
For every $n \in \N$, there exists $k = k(n) \in \N$ such that
$t_k \leq n < t_{k+1}$, and $\lim_{n \to \infty} k(n) = \infty$;
therefore, for all $n$ large enough such that $k \geq K$, it holds that
\[
X_n \leq \sup_{0 \leq m \leq t_{k+1}} X_m \leq t_{k+1}^{\chi + 2\eps} \leq \bigl( t_{k+1} / t_k \bigr)^{\chi + {2\eps}} n^{\chi + 2\eps} \leq C n^{\chi + 2\eps},
\]
where $C := \sup_{k \in \N} (t_{k+1}/t_k)^{\chi+2\eps} < \infty$.
Thus 
$\limsup_{n\to\infty} \frac{\log X_n }{\log n} \leq \chi + 2 \eps$, a.s., 
and since $\eps >0$ was arbitrary, we obtain~\eqref{eq:Xn-limsup}.
\end{proof}

We work towards the proof of Proposition~\ref{prop:nu-moments}.
The idea is to iterate a contraction argument
that gives bounds of the form
$\Exp (X_n^\nu) = O( n^{\theta_k} )$
for a sequence $\theta_k \to \alpha \chi$.
The following is the basis of the iteration scheme.

\begin{lemma}
    \label{lem:moment-iteration}
    Suppose that $\alpha \in (-1,\infty)$, $\beta, \gamma \in \RP$, and that~\eqref{ass:bounded-jumps} and~\eqref{ass:martingale} hold. Then
There exists a constant $\Lambda_0 < \infty$ such that, a.s.,
\begin{equation}
\label{eq:exponential-confinement}
X_n \leq \Lambda_0^n \max (1 , X_0), \text{ for all } n \in \ZP.
\end{equation}
Suppose that $1 +\gamma  > \alpha + 2 \beta$
and $\nu > \max ( \alpha +1, 2)$. Then $\Exp ( X_n ^\nu ) <\infty$ for all $n \in \ZP$, and 
there exists $C \in \RP$ such that, for all $n \in \ZP$,
\begin{equation}
    \label{eq:moment-iteration}
\Exp \big[X_{n+1}^{\nu} - X_n^{\nu}\big]
\leq 
C (1+n)^{\frac{\gamma}{2} - \beta}  \bigl( \Exp [X_n^{\nu}] \bigr)^{\frac{\nu-\alpha -1}{\nu}} 
+ C  \bigl( \Exp [X_n^{\nu}]\big)^{\frac{\nu-2}{\nu}}  + C.
\end{equation}
\end{lemma}
\begin{proof}
 Corollary~\ref{cor:confinement-bound} 
shows that $X_{n+1} \leq C \max (1, X_n)$, for all $n \in \ZP$, and we may suppose that $C > 1$.
Then~\eqref{eq:exponential-confinement} follows by an induction,
since if $X_n \leq C^n \max (1,X_0)$,
then $X_{n+1} \leq C \max (1, C^n , C^n X_0 ) \leq C^{n+1} \max (1, X_0)$.
This verifies~\eqref{eq:exponential-confinement}, and, moreover, shows that $\Exp (X_n^\nu ) < \infty$ for every $\nu \in \RP$ and every $n \in \ZP$ (recall that $X_0$ is fixed).

    Let $\nu > \max ( \alpha +1, 2)$. 
Combining~\eqref{eq:mean-increment-nu-upper} and~\eqref{eq:mean-increment-nu-upper-small} shows that, for some $C < \infty$,
\begin{equation}
\label{eq:start-bound}
\Exp (X_{n+1}^{\nu} - X_n^{\nu} \mid \cF_n )
 \leq C X_n^{\nu -1} \kappa_n (Z_n)    + C  X_n^{\nu-2}  + C, \text{ for all } n \in \ZP.
\end{equation}
From~\eqref{eq:kappa-def} and H\"older's inequality, we obtain that for any  $p,q>1$ such that $p^{-1} + q^{-1} = 1$,
\begin{align*}
\Exp \bigl( X_n^{\nu -1} \kappa_n (Z_n) \bigr)
&\leq  \rho  (1+n)^{-\beta} \Exp \bigl( X_n^{\nu-\alpha -1}|Y_n|^\gamma \bigr) \\
&\leq \rho (1 + n)^{-\beta} \big(\Exp[X_n^{(\nu-\alpha -1)p}]\big)^{1/p} \bigl( \Exp [|Y_n|^{\gamma q}] \bigr)^{1/q}.
\end{align*}
Now by Lemma~\ref{lem:Y-moments} for every $q>0$ we have that $\Exp[|Y_n|^{\gamma q}]^{1/q} \leq A^{1/q}_{\gamma q} n^{\gamma/2}$ and therefore, for every $p>1$ there is $C>0$ for which we have
\[
\Exp \bigl( X_n^{\nu -1} \kappa_n (Z_n) \bigr)
\leq  C (1+n)^{\frac{\gamma}{2} - \beta}  \bigl( \Exp [X_n^{(\nu-\alpha -1)p}] \bigr)^{1/p} ,
\text{ for all } n \in \ZP.
\]
Since $\nu > 1+\alpha > 0$, we may (and do) take $p = \nu/(\nu - \alpha -1) >1$,
so that  
\begin{equation}
\label{eq:holder-induction}
\Exp \bigl( X_n^{\nu -1} \kappa_n (Z_n) \bigr)
\leq 
C (1+n)^{\frac{\gamma}{2} - \beta}  \bigl( \Exp [X_n^{\nu}] \bigr)^{\frac{\nu-\alpha -1}{\nu}} ,
\text{ for all } n \in \ZP.
\end{equation}
Using the bound~\eqref{eq:holder-induction} in~\eqref{eq:start-bound},
and taking expectations, 
 we obtain, for some $C < \infty$, 
\[
\Exp \big[X_{n+1}^{\nu} - X_n^{\nu}\big]
\leq 
C (1+n)^{\frac{\gamma}{2} - \beta}  \bigl( \Exp [X_n^{\nu}] \bigr)^{\frac{\nu-\alpha -1}{\nu}} 
+ C   \Exp ( X_n^{\nu -2} )  + C,
\]
and then using the bound 
$\bigl( \Exp ( X_n^{\nu -2} ) \bigr)^\nu \leq \bigl( \Exp ( X_n^{\nu} ) \big)^{{\nu-2}}$ for $\nu \geq 2$,
which follows from Jensen's inequality, we obtain~\eqref{eq:moment-iteration}.
\end{proof}

The next step is to find a starting value $\theta_0$
for the iteration scheme.

\begin{lemma} \label{lem:first-bound}
Suppose that $\alpha \in (-1,\infty)$,   $\beta, \gamma \in \RP$, and that~\eqref{ass:bounded-jumps} and~\eqref{ass:martingale} hold.
Suppose that $1 +\gamma  > \alpha + 2 \beta$. For every $\nu > \max (1+\alpha, 2)$, there exists $\theta_0 \in \RP$ for which 
  $\Exp ( X^{\nu}_n ) - \Exp (X^\nu_0 ) \leq C  n^{\theta_0}$ for all $n \in \ZP$.
\end{lemma}
\begin{proof}
Fix $\nu > \max (1+\alpha, 2)$.
We improve the exponential bound in~\eqref{eq:exponential-confinement} to obtain a polynomial bound.
For ease of notation, write $f_n := \Exp ( X_n^\nu )$,
which is finite by Lemma~\ref{lem:moment-iteration}.
Set $\lambda := \max\{(\nu-\alpha-1)/\nu, (\nu-2)/\nu\}$, which satisfies $\lambda \in (0,1)$
since $\nu > \max (1+\alpha, 2)$ and $\alpha > -1$. Then from~\eqref{eq:moment-iteration},
it follows that, for some $C < \infty$,
\begin{equation}
\label{eq:dif-theta-contract-f}
f_{n+1} - f_n \leq C  \bigl( 1 + n^a f_n^\lambda \bigr) , \text{ for all } n \in \ZP,
\end{equation}
where $a :=\max (\frac{\gamma}{2} -\beta,0 )$. 
It follows from~\eqref{eq:dif-theta-contract-f} that
\begin{equation}
\label{eq:base-step-1}
f_{n} - f_0 \leq \sum_{k = 0}^{n-1} \bigl( f_{k+1}-f_k \bigr) 
\leq \sum_{k =0}^{n-1} C \bigl(1 + k^a f_k^\lambda \bigr) \leq Cn +C n^{a+1} \max_{0 \leq k \leq n-1} f_k^\lambda.
\end{equation}
From~\eqref{eq:exponential-confinement} we have that 
there is $\Lambda_0>1$ (depending on $X_0$)
such that for all $n \in \ZP$,
$f_n \leq \Lambda_0^n$.
With this, we come back to~\eqref{eq:base-step-1} and obtain that,
for $D \geq f_0 + 2 C \in \RP$,
\begin{align}
\label{eq:base-step-2}
f_{n} &
\leq f_0 + C n + C n^{a+1} \Lambda_0^{n \lambda} \nonumber\\
& \leq D n^{a+1} \Lambda_0^{n\lambda}, \text{ for all } n \in \N.
\end{align}
Since $\lambda<1$, we may choose 
$D \geq \max ( 1, f_0 + 2 C )$
large enough so that
$(f_0 + 2 C )D^\lambda \leq D$;
we may then assume
that~$D$ in~\eqref{eq:base-step-2} has this property.
We  claim that for any $j, n \in \N$,
\begin{equation}\label{eq:j-bound}
f_n \leq D n^{(a +1)\sum_{i = 0}^{j-1}\lambda^i } \Lambda_0^{n \lambda^j}.
\end{equation}
The bound obtained in \eqref{eq:base-step-2} verifies the case $j =1$
of~\eqref{eq:j-bound}. For an induction,  assume that~\eqref{eq:j-bound}
holds for a given $j \in \N$. Then applying that bound
in~\eqref{eq:dif-theta-contract-f}, we obtain
\begin{align*}
f_{n} & \leq f_0 + C n + C n^{a+1} D^\lambda n^{\lambda (a+1) \sum_{i=0}^{j-1} \lambda^i} \Lambda_0^{n \lambda^{j+1}} \\
& \leq  \bigl( f_0 + 2 C \bigr)   D^\lambda n^{(a+1) \sum_{i=0}^{j} \lambda^i} \Lambda_0^{n \lambda^{j+1}} ,\end{align*}
using the fact that $D^\lambda n^{(a+1) \sum_{i=0}^{j} \lambda^i} \geq n \geq 1$. 
Since $(f_0 +2C) D^\lambda \leq D$,
we obtain the $j+1$ case of~\eqref{eq:j-bound}, completing the induction.
Since $\lambda <1$, we have $\sum_{i=0}^{j-1} \lambda^i \leq (1-\lambda)^{-1} < \infty$ for all~$j$, we conclude from~\eqref{eq:j-bound} that, if we set $\theta_0 := (a+1)(1-\lambda)^{-1}$,
\[ \sup_{n \in \N} \left[ \frac{1}{n} \log \left( \frac{f_n}
{D n^{\theta_0}} \right) \right] \leq \lambda^j  \log \Lambda_0 . \]
This holds for all $j \in \N$, and, since $\lambda <1$,
we can take $j \to \infty$ to obtain $f_n \leq D n^{\theta_0}$ for all $n \in \N$, as required.
\end{proof}

The iterative scheme for improving the exponent $\theta_0$
through a sequence of exponents~$\theta_k$ 
is defined via
two functions $F, G: \RP \to \R$  given by
\begin{align}
    \label{eq:kappa-recursion-F-G}
F (\theta) :=  \frac{\gamma}{2}-\beta + \left( \frac{\nu-1-\alpha}{\nu} \right) \theta , \text{ and } G (\theta ) :=   \left(\frac{\nu-2}{\nu} \right) \theta .
\end{align}
The following result is the basis for the success of the iteration scheme; its proof comes at the end of this section. Write $a \vee b := \max (a,b)$
for $a,b \in \R$.

\begin{lemma}
\label{lem:kappa}
Suppose that 
$\alpha > -1$, $\gamma \in \RP$ and $\beta \in \R$ are such that
$1 +\gamma  > \alpha + 2 \beta$ and that $\nu > \max (2, 1 + \alpha)$.
Let $\theta_0 > \nu \chi$, and define $\theta_k$, $k \in \N$, by the recursion
\begin{align}
    \label{eq:kappa-recursion}
\theta_{k+1} & := 1 + 
\bigl( F (\theta_k) \vee G (\theta_k ) \bigr), 
\text{ for } k \in \ZP, 
\end{align}
where $F, G: \RP \to \R$ are defined by~\eqref{eq:kappa-recursion-F-G}.
 Then, 
$\lim_{k \to \infty} \theta_k =  \nu \chi$.
\end{lemma}

 \begin{proof}[Proof of Proposition~\ref{prop:nu-moments}]
 An appeal to Jensen's inequality shows that   to
 prove~\eqref{eq:x-nu-moments} it suffices to suppose that $\nu > \max (2, 1 +\alpha)$. 
We will define a non-increasing sequence of positive exponents $\theta_0, \theta_1, \ldots$.
Suppose we know that $\Exp ( X^{\nu}_n ) - \Exp (X^\nu_0 ) \leq C_k n^{\theta_k}$
for some $C_k < \infty$ and all $n \in \N$; this holds with $k=0$ and some $\theta_0 > \chi \nu$, by Lemma~\ref{lem:first-bound}.
Then, from~\eqref{eq:moment-iteration} and the hypothesis on $\Exp ( X^{\nu}_n )$,
we obtain
\begin{equation}
    \label{eq:X-increment-bound-2}
\Exp \bigl( X_{n+1}^{\nu} - X_n^{\nu}   \bigr)
\leq C n^{\frac{\gamma}{2}-\beta + ( \frac{\nu-1-\alpha}{\nu} ) \theta_k}  + C n^{(\frac{\nu-2}{\nu}) \theta_k} + C.
\end{equation}
It follows from~\eqref{eq:X-increment-bound-2} that
 $\Exp ( X_n^{\nu} ) - \Exp ( X_0^\nu ) \leq C_{k+1} n^{\theta_{k+1}}$,
 where $C_{k+1} < \infty$ and $\theta_{k+1}$
 is given by~\eqref{eq:kappa-recursion}.
 Lemma~\ref{lem:kappa} shows
that $\lim_{k \to \infty} \theta_k =  \nu \chi$ and thus completes the proof.
\end{proof}

\begin{proof}[Proof of Lemma~\ref{lem:kappa}]
Note that, by~\eqref{eq:chi-def}, it holds that
\begin{equation}
    \label{eq:F-nu-chi}
    F ( \nu \chi ) 
    = \frac{\gamma}{2} -\beta  + \nu \chi - (1 +\alpha ) \chi
    =   \nu \chi - 1.
\end{equation}
We claim that 
\begin{equation}
    \label{eq:kappa-lower-bound}
    1 + \bigl( F(\theta ) \vee G(\theta) \bigr) \geq \nu \chi, \text{ whenever } \theta \geq \nu \chi.
\end{equation}
Indeed, since $\nu \geq 1+\alpha$ we have,
whenever $\theta \geq \nu \chi$,
$1 + F(\theta) \geq  1 + F (\nu\chi) = \nu \chi$, by~\eqref{eq:F-nu-chi}.
This verifies~\eqref{eq:kappa-lower-bound}. Consequently, started from $\theta_0 > \nu \chi$, 
\begin{equation}
    \label{eq:kappa-k-lower-bound}
    \theta_k \geq \nu \chi, \text{ for all } k \in \ZP.
\end{equation}
Next we claim that there is a constant $c_\alpha >0$ such that, for every $\eps \geq 0$,
\begin{equation}
    \label{eq:kappa-decreases}
    1 + \bigl( F(\theta ) \vee G(\theta) \bigr) \leq \theta - c_\alpha \eps, \text{ whenever } \theta \geq \nu \chi+\eps.
\end{equation}
Assume~\eqref{eq:kappa-decreases} for now;
combined with~\eqref{eq:kappa-k-lower-bound}, it follows from~\eqref{eq:kappa-decreases} that
\begin{equation}
    \label{eq:kappa-k-decreases}
    \theta_{k+1} \leq \theta_k, \text{ for every } k \in \ZP.
\end{equation}
Recall that the hypothesis $\gamma + 1 > \alpha + 2 \beta$
means that $\chi > 1/2$. If $\theta \geq \nu \chi +\eps$,
then 
\[ \nu \bigl( 1 + F(\theta) -\theta \bigr)
= \nu \left( 1 + \frac{\gamma}{2} -\beta \right) - (1+\alpha) \theta
=  (1+\alpha ) \bigl( \nu \chi - \theta ) \leq - \eps (1+\alpha). 
\]
Similarly, for $\theta \geq \nu \chi+\eps$,
\[ \nu \bigl( 1 + G(\theta ) - \theta \bigr)
= \nu - 2 \theta = (1-2\chi) \nu + 2 (\chi \nu - \theta) \leq - 2\eps, \]
since $\chi > 1/2$.
This verifies~\eqref{eq:kappa-decreases} with $c_\alpha := \min ({2/\nu}, {(1+\alpha)/\nu} ) \in (0,1)$. Fix $\eps>0$ and define $k_\eps := \inf \{ k \in \ZP : \theta \leq \nu \chi + \eps\}$.
It follows from~\eqref{eq:kappa-decreases} that $\theta_{k+1} \leq \theta_k- c_\alpha \eps$ whenever $k \leq k_\eps$, and hence $k_\eps \in \ZP$ is finite. 
Together with~\eqref{eq:kappa-k-decreases}, we have thus shown
that $\limsup_{k \to \infty} \theta_k \leq \nu \chi + \eps$.
Since $\eps>0$ was arbitrary, this means that
$\limsup_{k \to \infty} \theta_k \leq \nu \chi$.
Combined with~\eqref{eq:kappa-k-lower-bound}, we get the claimed result.
 \end{proof}

 \subsection{The lower bound}
 \label{sec:lower-bound}

Corollary~\ref{cor:X-limsup} gives the `$\limsup$' half of
Theorem~\ref{thm:main-thm}; it remains to verify the `$\liminf$' half.
This is the purpose of this section. Recall that $\alpha^+ = \max (0, \alpha)$.

\begin{proposition}
\label{prop:X-lower-bound}
Suppose that~\eqref{ass:bounded-jumps}, \eqref{ass:martingale}, and~\eqref{ass:ellipticity} hold.  Suppose that $\alpha \in (-1,\infty)$, and $\beta, \gamma \in \RP$ are such 
that $1 + \gamma > \alpha^+ + 2 \beta$.
Then 
\begin{equation}
    \label{eq:Xn-liminf}
\liminf_{n \to \infty} 
\frac{\log X_n}{\log n} \geq \chi, \as \end{equation}
\end{proposition}

Define  processes $A := (A_n, n \in \ZP)$ and $\Xi:= (\Xi_n, n \in \ZP)$ on $\R$ by $A_0 = \Xi_0 = 0$, 
and
\begin{equation}
\label{eq:An-Mn-def}
A_n:=  \sum_{m=0}^{n-1} \kappa_m (Z_m) , ~\text{and}~
\Xi_n := \sum_{m=0}^{n-1}  \xi^{(1)}_{m+1} , \text{ for } n \in \N.
\end{equation}
From~\eqref{eq:X-increment} and the fact that $A_n \geq 0$, we then observe that
\begin{equation}
    \label{eq:doob-concrete}
X_n - X_0 = A_n + \Xi_n \geq \Xi_n, \text{ for all } n \in \ZP;
\end{equation}
while~\eqref{eq:doob-concrete} is reminiscent of the 
Doob decomposition for the submartingale~$X$,
and the previsible process $A$ is non-decreasing, 
the process $\Xi$
is itself a submartingale, by condition~\eqref{eq:xi-martingale}.
The next result controls the size of $\Xi_n$ as defined at~\eqref{eq:An-Mn-def}.

\begin{proposition}
\label{prop:mart-ab}
Suppose that~\eqref{ass:bounded-jumps} and~\eqref{ass:martingale} hold. 
Then  
 \[ \limsup_{n \to \infty} \frac{\log \max_{0 \leq m \leq n} |\Xi_m|}{\log n} \leq \frac{1}{2}, \as\]
\end{proposition}
\begin{proof}
From the definition of $\Xi_n$ at~\eqref{eq:An-Mn-def} and~\eqref{eq:xi-jump-bound} and~\eqref{eq:xi-martingale}, we have
\begin{align}
\label{eq:martingale-square-increment}
\Exp \bigl[ \Xi_{n+1}^2-\Xi_n^2 \bigmid \cF_{n} \bigr] & =
\Exp \bigl[ ( \Xi_{n+1} - \Xi_n)^2 \bigmid \cF_n \bigr] 
+ 2 \Xi_n \Exp  \bigl[ \Xi_{n+1}-\Xi_n \bigmid \cF_{n} \bigr] \nonumber\\ 
& \leq B^2 + 2 B \Xi_n \1{ X_n \leq B} \leq 3 B^2, \as,
 \end{align}
 since $\Xi_n \leq X_n$.
Taking expectations in~\eqref{eq:martingale-square-increment} and using the fact that $\Xi_0 = 0$, we  obtain
\begin{equation}
    \label{eq:M-qv-bound}
 \Exp (\Xi_n^2)
= \sum_{m=0}^{n-1} \Exp ( \Xi_{m+1}^2 - \Xi_m^2 )
\leq 3 B^2 n, \text{ for all } n\in \ZP. \end{equation}
Since $\Xi_n^2$ is a non-negative submartingale,
Doob's maximal inequality together with~\eqref{eq:M-qv-bound}
implies that, for every $\eps >0$ and every $n \in \ZP$,
\[ \Pr \left( \max_{0 \leq m \leq n} | \Xi_m | \geq n^{(1/2)+\eps} \right) 
\leq n^{-1-2\eps} \Exp ( \Xi_n^2 ) = O ( n^{-2\eps} ) .\]
In particular, the Borel--Cantelli lemma shows that there are only finitely many $k \in \ZP$ such that
$\max_{0 \leq m \leq 2^k} | \Xi_m | \geq (2^k)^{(1/2)+\eps}$. Any $n \in \N$ has $2^{k_n} \leq n < 2^{k_n+1}$
with $k_n \to \infty$ as $n \to \infty$, so, for all but finitely many $n \in \ZP$,
\[ \max_{0 \leq m \leq n} | \Xi_m | \leq \max_{0 \leq m \leq 2^{k_n+1}} | \Xi_m | \leq  (2 \cdot 2^{k_n})^{(1/2)+\eps} \leq ( 2 \cdot n)^{(1/2)+\eps} .\]
Since $\eps>0$ was arbitrary, this completes the proof.
\end{proof}

\begin{proof}[Proof of Proposition~\ref{prop:X-lower-bound}]
Let $\beta , \gamma \in \RP$.
First suppose that $\alpha \geq 0$. Let $\eps >0$.
Then the upper bound 
$X_n \leq n^{\chi+\eps}$, for all but finitely many~$n \in \N$, a.s.,
established in Corollary~\ref{cor:X-limsup}, together with~\eqref{eq:kappa-def}, shows that
\[ \kappa_n (Z_n) = \frac{\rho |Y_n|^\gamma}{(1+n)^\beta (1+X_n)^{\alpha}}
\geq \frac{\rho |Y_n|^\gamma}{(1+n)^{\beta+\alpha \chi + \alpha \eps}} ,
\]
for all but finitely many $n$. It follows that, for some $c>0$ and all $n \in \N$ large enough,
\begin{equation}
    \label{eq:A-lower-alpha-nonneg}
 A_n = \sum_{m=0}^{n-1} \kappa (Z_m ) \geq  \rho {(1+ n)}^{-\alpha (\chi+\eps)} \sum_{m=0}^{n-1} (1+m)^{-\beta} | Y_m |^\gamma \geq c n^{ \chi -  \alpha \eps -\eps} ,
\end{equation}
by~\eqref{eq:chi-def} and~\eqref{eq:|Y|-limit}. By hypothesis,
$(1+\alpha) \chi = 1 + \frac{\gamma}{2} - \beta > \frac{1+\alpha}{2}$,
so that $\chi > 1/2$; since $\eps >0$ was arbitrary, we can combine~\eqref{eq:doob-concrete} and~\eqref{eq:A-lower-alpha-nonneg} with 
the bound on $\Xi_n$
from Proposition~\ref{prop:mart-ab} to deduce~\eqref{eq:Xn-liminf}.

Next suppose that $\alpha \in (-1,0)$ and $1 + \gamma > 2 \beta$. 
We will prove, by an induction on $k \in \N$,
 that
\begin{equation}
    \label{eq:Xn-liminf-induction}
\liminf_{n \to \infty}
\frac{\log X_n}{\log n} \geq \chi_k, \as, \end{equation}
 where, using~\eqref{eq:chi-def},  we write
 \begin{equation}
     \label{eq:chi-k-def}
     \chi_k := \left( 1 + \frac{\gamma}{2} - \beta \right) \sum_{\ell=0}^{k-1} (-\alpha)^\ell 
     = \chi \bigl( 1 - (-\alpha)^k \bigr ), \text{ for } k \in \N.
 \end{equation}
 By~\eqref{eq:chi-k-def} and the fact that $\alpha \in (-1,0)$, for all $k \in \N$ it holds that
$\chi_{k+1} > \chi_k$, while, by~\eqref{eq:chi-def},
 $\chi_1 = (1+\alpha) \chi = 1 + \frac{\gamma}{2} - \beta > 1/2$, by the hypothesis $1+\gamma > 2 \beta$.
 Hence $\chi_k > 1/2$ for all $k \in \N$.
 
 To start the induction, note that, since $\alpha < 0$,
$\kappa_n (x,y) \geq \rho (1+n)^{-\beta} y^\gamma$, and so
\[ A_n \geq \rho \sum_{m=0}^{n-1} (1+m)^{-\beta} |Y_m|^\gamma .\]
By Proposition~\ref{prop:|Y|-sum}, this shows that, for every $\eps >0$, a.s.,
$A_n \geq n^{\chi_1 - \eps}$ for all but finitely many $n \in \N$,
where $\chi_1 = 1 + \frac{\gamma}{2} - \beta$ as given by the $k=1$ case of~\eqref{eq:chi-k-def}.
Since $\chi_1 > 1/2$, and $\eps >0$ was arbitrary, we can combine this with 
the bound on $\Xi_n$
from Proposition~\ref{prop:mart-ab} to see that, for every $\eps >0$, a.s.,
$X_n \geq n^{\chi_1 - \eps}$ for all but finitely many $n \in \N$. This verifies~\eqref{eq:Xn-liminf-induction}
for $k=1$.

 For the inductive step, suppose that~\eqref{eq:Xn-liminf-induction} holds for a given $k \in \N$.
Then, for every $\eps \in (0,\chi_k)$, $X_n \geq n^{\chi_k - \eps}$
for all but finitely many~$n \in \N$.
Hence
\[ \kappa_n (Z_n) = \frac{\rho |Y_n|^\gamma}{(1+n)^\beta} (1+X_n)^{|\alpha|}
\geq \frac{\rho |Y_n|^\gamma}{(1+n)^{\beta}} n^{|\alpha| ( \chi_k -\eps)} .
\]
By Proposition~\ref{prop:|Y|-sum}, applied with $\beta + \alpha (\chi_k - \eps)$ in place of $\beta$, 
we obtain
\[ A_n = \sum_{m=0}^{n-1} \kappa (Z_m ) \geq  \rho  \sum_{m=0}^{n-1} (1+m)^{|\alpha| (\chi_k - \eps) -\beta} | Y_m |^\gamma \geq n^{1 + \frac{\gamma}{2} - \beta - \alpha \chi_k + \alpha \eps - \eps}, \]
for all but finitely many $n \in \N$, a.s. That is,
\begin{align*} \liminf_{n \to \infty} \frac{\log A_n}{\log n} & \geq 1 + \frac{\gamma}{2} - \beta - \alpha \chi_k + \alpha \eps - \eps \\
& = (1 +\alpha) \chi - \alpha \bigl( 1 - (-\alpha)^k \bigr) \chi + \alpha\eps - \eps  = \chi_{k+1}  + \alpha \eps - \eps ,\end{align*}
by~\eqref{eq:chi-k-def} and~\eqref{eq:chi-def}. Since $\chi_{k+1} > 1/2$, and $\eps>0$ was arbitrary, we can combine this with 
the bound on $\Xi_n$
from Proposition~\ref{prop:mart-ab} to verify the $k+1$ case of~\eqref{eq:Xn-liminf-induction}.

We have established that~\eqref{eq:Xn-liminf-induction} holds for arbitrary $k \in \N$. Moreover, 
by taking $k \to \infty$ in~\eqref{eq:chi-k-def} and using the fact that $|\alpha|<1$ and~\eqref{eq:chi-def}, 
we obtain $\lim_{k \to \infty} \chi_k = \chi$, and hence deduce~\eqref{eq:Xn-liminf}
from~\eqref{eq:Xn-liminf-induction}.
\end{proof}

\begin{proof}[Proof of Theorem~\ref{thm:main-thm}]
The theorem combines the `$\limsup$' result from Corollary~\ref{cor:X-limsup}
with the `$\liminf$' result from Proposition~\ref{prop:X-lower-bound}.    
\end{proof}

\section{A barycentric excluded-volume random walk}
\label{sec:barycentre}

\subsection{A brief history}

The model $(W_n,G_n)$
described informally in Section~\ref{sec:intro}
dates back to 2008, when
some of the present authors
discussed a number of self-interacting walk models with Francis Comets.
The common feature of these models is they are random walk models with self-interaction
mediated by some global geometric functionals of the past trajectory, or, equivalently, its
occupation times. The three models we discussed with Francis are as follows.
\begin{itemize}
\item The paper~\cite{cmvw} studied a random walk
that is either attracted to or repelled by the centre of mass of its previous trajectory, with an interaction
strength that decays with distance.
\item The random walk that avoids its convex hull had been introduced by
 Angel, Benjamini, and Vir\'ag, 
and studied in~\cite{abv,zern}.
The original model is conjectured to be ballistic (and this conjecture is not yet fully settled),
but~\cite{cmw} 
introduced a version of the model which avoids not the full convex hull, but the convex hull of the starting point and the most recent~$k$ locations, and established that the finite-$k$ model is ballistic.
\item The barycentric excluded-volume model $(W_n, G_n)$ that we discuss here, and for which the behaviour remains entirely open. The heuristic identification of the~$3/4$ exponent presented in Section~\ref{sec:intro} was sketched
out in 2008--9 with Francis and Iain MacPhee.
\end{itemize}

\subsection{An heuristic link between the two models}
\label{sec:heuristic}

For $n \in \ZP$, let $\cF_n$ be the $\sigma$-algebra $\sigma (W_0, \ldots , W_n)$,
so that both $W_n$ and $G_n$ are $\cF_n$-measurable, as is $T_n := W_n - G_n$.

If $\| W_n\| \| T_n \| > 0$, define 
\[ v_n := \frac{\hW_n + \hT_n}{\bigl\|\hW_n + \hT_n\bigr\|} ,\]
where $\widehat{u}: = u/\|u\|$,
and set $v_n := 0$ otherwise. Let $v_n^\perp$ be any unit vector with $v_n \cdot v_n^\perp = 0$.
If $\| W_n\| \| T_n \| > 0$, define
\[ \beta_n := \frac{1}{2} \arccos \left( \frac{ W_n \cdot T_n}{ \| W_n\| \| T_n \|} \right) ,\]
and set $\beta_n := 0$ otherwise. Here, the (principal branch of the) arc-cosine function $\arccos : [-1,1] \to [0,\pi]$ is given by
\begin{equation}
    \label{eq:arccos}
 \arccos \lambda := \int_\lambda^1 \frac{\ud t}{\sqrt{1-t^2}}  , \text{ for } -1 \leq \lambda \leq 1. \end{equation}
Define $\cC_n \subset \R^2$ as
\[ \cC_n := \left\{ z \in \R^2 :  {(z - W_n)} \cdot v_n < -  \beta_n \| z - W_n \|  \right\}. \]
If $\| W_n\| \| T_n \| > 0$, then $\cC_n$ is a (non-degenerate, open) cone with apex $W_n$, boundary given by semi-infinite rays from $W_n$ in directions
$-W_n$ and $-T_n$, 
and angular span~$2 \beta_n$; if $\| W_n \| \|T_n \| =0$, then $\cC_n = \emptyset$. See Figure~\ref{fig:flory-XY} for a picture.

 We take the distribution of $W_{n+1}$ given
   $\cF_n$ to be uniform
   on the unit-radius circle centred at $W_n$
   excluding the arc intersecting~$\cC_n$. 
   More formally, we can generate the process via
   a sequence $U_1, U_2, \ldots$ of independent $U[-1,1]$ variables as follows. Given $W_n$ and $G_n$,
   set
  \begin{align*} W_{n+1} - W_n = v_n  \cos \bigl( (\pi - \beta_n ) U_{n+1} \bigr) 
  + v_n^\perp \sin \bigl( (\pi -  \beta_n  ) U_{n+1} \bigr) , \text{ for all } n \in \ZP. \end{align*}

So far analysis of this model has eluded us. As described in Section~\ref{sec:intro},
we have heuristic and numerical evidence in support of the following.

	\begin{conjecture}
	\label{conj:scaling}
For the $(W_n, G_n)$ model, the asymptotics at~\eqref{eq:flory-scaling} hold.
	\end{conjecture}

A calculation shows that
\[ \Exp ( W_{n+1} - W_n \mid \cF_n ) = \frac{2 \sin \beta_n}{\pi - \beta_n} v_n .\]
Now, it is slightly more convenient for comparison with the model in Section~\ref{sec:model}
to ``symmetrize'' the process $W$ so that the support of the increment excludes
not only the cone $\cC_n$, but also
its reflection through the line from $0$ to $W_n$. The upshot of this is to double that excluded angle, and to ensure that the mean drift is in the $\hW_n$ direction:
\[ \Exp ( W_{n+1} - W_n \mid \cF_n ) = \frac{2 \sin (2 \beta_n)}{\pi - 2\beta_n} \hW_n .\]
Now define
\[ X_n = \| W_n \|, \text{ and } | Y_n|  = \| W_n \| \tan (2 \beta_n) ; \]
see Figure~\ref{fig:flory-XY}.
Most of the time, we expect $\beta_n$ to be very small (we believe that  $\lim_{n \to \infty} \beta_n = 0$, a.s.)
and so $\sin (2\beta_n) \approx 2 \beta_n \approx \tan (2 \beta_n)$.
Hence we claim that, roughly speaking,
\begin{equation}
    \label{eq:34-approx}
 \Exp ( W_{n+1} - W_n \mid \cF_n ) \approx \frac{2}{\pi} \frac{|Y_n|}{1+X_n} \hW_n .\end{equation}
 The analogy between~\eqref{eq:34-approx}
 and the $\beta =0$, $\alpha = \gamma =1$ model of Example~\ref{ex:34} 
 is now clear enough, but far from exact. Nevertheless, 
 we believe that~\eqref{eq:34-approx} is strongly suggestive that
 the conclusion of Example~\ref{ex:34} would also hold in the present setting, i.e., $\log \| W_n \| / \log n \to 3/4$. On the basis that the fluctuations of $|Y_n|$ are diffusive, it is then natural to conjecture the limiting direction, $\hW_n \to \Theta$, as in~\eqref{eq:flory-scaling}: cf.~\cite[\S 4.4.3]{mpw}.

\begin{figure}
\begin{center}
\begin{tikzpicture}[domain=0:10, scale = 1.4,decoration={
    markings,
    mark=at position 0.6 with {\arrow[scale=2]{>}}}]
\filldraw (0,0) circle (2pt);
\filldraw (3,1) circle (2pt);
\filldraw (5,0) circle (2pt);
\draw (0,0) -- (5,0);
\draw[dashed] (0,2.5) -- (5,0);
\draw[dashed] (0,2.5) -- (0,0);
\draw[dashed] (4,-0.5) -- (5,0);
\draw[dashed,<->] (-0.5,2.5) -- (-0.5,0);
\draw[dashed,<->] (0,3.0) -- (5,3.0);
\draw (3,1) -- (5,0);
\node at (-0.2,0)       {$0$};
\node at (5.1,-0.3)       {$W_n$};
\node at (3,1.3)       {$G_n$};
\node at (3.7,0.3)       {$2 \beta_n$};
\node at (-1.0,1.25)       {$|Y_n|$};
\node at (2.5,3.3)       {$X_n$};
\draw ({5+1.0*cos(180)},{0+1.0*sin(180)}) arc (180:207:1.0);
\draw ({5+1.0*cos(180)},{0+1.0*sin(180)}) arc (180:153:1.0);
\draw[double] ({5+0.6*cos(153)},{0+0.6*sin(153)}) arc (153:{207-360}:0.6);
\end{tikzpicture}
\end{center}
\caption{\label{fig:flory-XY} Associating processes $X, Y$ to the processes $W, G$. Depicted here is the symmetrized version of the model, where excluded is not only the triangle with angle~$2\beta_n$, but also its reflection in the line through $0$ and~$W_n$.}
\end{figure}

\appendix

\section{Path-sums of one-dimensional martingales}
\label{sec:one-dimensional-bounds}

On a probability space $(\Omega,\cF,\Pr)$,
let $S := ( S_n, n \in \ZP)$ be an 
$\RP$-valued process adapted to a filtration $(\cF_n, n \in \ZP)$. 
For $\gamma \in \RP$, $\beta \in \R$, and $n \in \N$, define 
the path sum
\begin{equation}
    \label{eq:Gamma-def}
    \Gamma_n (\beta, \gamma ) := \sum_{m=1}^n m^{-\beta} S_m^\gamma.
    \end{equation}
Here is the main result of this section.

\begin{theorem}
    \label{thm:S-bounds}
Let $S := ( S_n, n \in \ZP)$ be an 
$\RP$-valued process adapted to a filtration $(\cF_n, n \in \ZP)$. 
Suppose that
there exist $B, x_0 \in \RP$   and $\delta >0$ such that, for all $n \in \ZP$,
\begin{align}
\label{eq:S-jump-bound}
\Pr ( | S_{n+1} - S_n | \leq B ) & = 1 ; \\
\label{eq:S-ellipticity}
\Pr (  S_{n+1} - S_n \geq \delta \mid \cF_n ) & \geq \delta ; \text{ and}\\
\label{eq:S-zero-drift}
0  \leq \Exp ( S_{n+1} - S_n \mid \cF_n )  & \leq B \1 { S_n < x_0 } .\end{align}
Then, 
\begin{equation}
\label{eq:S-diffusive} 
\lim_{n \to \infty} \frac{ \log \max_{1 \leq m \leq n} S_m}{\log n } =  \frac{1}{2} , \as \end{equation}
Moreover, suppose that $\gamma \in \RP$ and $\beta \in \R$.
Then, if 
$2\beta \leq 2+\gamma$,
\begin{equation}
\label{eq:S-limit} 
\lim_{n \to \infty} \frac{ \log \Gamma_n (\beta, \gamma )}{\log n } = 1 + \frac{\gamma}{2} - \beta, \as \end{equation}
On the other hand, if $2\beta > 2 +\gamma$, then $\sup_{n \in \N} \Gamma_n (\beta, \gamma ) < \infty$, a.s.
\end{theorem}

We prove Theorem~\ref{thm:S-bounds} in the rest of this section.
First, the upper bounds
are straightforward consequences of a maximal inequality; the statement is the next result.

\begin{lemma}
    \label{lem:S-upper-bounds}
Let $S := ( S_n, n \in \ZP)$ be an 
$\RP$-valued process adapted to a filtration $(\cF_n, n \in \ZP)$. 
Suppose that
there exist $B, x_0 \in \RP$ such that~\eqref{eq:S-jump-bound}
and~\eqref{eq:S-zero-drift} hold.
Then, 
\begin{equation}
\label{eq:S-diffusive-upper-bound} 
 \limsup_{n \to \infty} \frac{\log \max_{1 \leq m \leq n} S_m}{\log n} \leq \frac{1}{2}, \as \end{equation}
    \end{lemma}

Before giving the proof of Lemma~\ref{lem:S-upper-bounds},
we make some preliminary calculations. 
Write $D_n :=  S_{n+1} - S_n$. Note that~\eqref{eq:S-jump-bound} says that $|D_n| \leq B$, a.s.,
while~\eqref{eq:S-zero-drift} implies that $0 \leq  \Exp ( D_n \mid \cF_n ) \leq B \1 {S_n < x_0 }$, a.s.
Hence, for every $n \in \ZP$, a.s.,
\begin{align}
\label{eq:S-sq-upper}
 \Exp ( S_{n+1}^2 - S_n^2 \mid \cF_n ) & = 2 S_n \Exp (D_n \mid \cF_n ) + \Exp ( D_n^2 \mid \cF_n ) \leq 2 B x_0 + B^2.
\end{align}
On the other hand, by~\eqref{eq:S-ellipticity},
\begin{equation}
    \label{eq:S-variance}
 \Exp ( D_n^2 \mid \cF_n ) \geq \delta^2 \Pr  (  D_n \geq \delta  \mid \cF_n ) \geq \delta^3, \as \end{equation}
Therefore, using the lower bound in~\eqref{eq:S-zero-drift} and~\eqref{eq:S-variance}, 
we obtain
\begin{align}
\label{eq:S-sq-lower}
 \Exp ( S_{n+1}^2 - S_n^2 \mid \cF_n ) & = 2 S_n \Exp ( D_n \mid \cF_n ) + \Exp ( D_n^2 \mid \cF_n ) 
 \geq \delta^3, \as,
\end{align}
a fact that we will use
in the proof of Lemma~\ref{lem:kolmogorov} below.
    
\begin{proof}[Proof of Lemma~\ref{lem:S-upper-bounds}]
The result follows from the bound~\eqref{eq:S-sq-upper}
on the increments of $S_n^2$ together with Doob's inequality and a standard subsequence argument:
concretely, one may apply Theorem~2.8.1 of~\cite[p.~78]{mpw} 
to obtain that for any $\eps>0$, 
a.s., for all but finitely many $n \in \N$, $\max_{0 \leq m \leq n} S_m \leq n^{(1/2)+\eps}$. Since $\eps>0$ was arbitrary, we get~\eqref{eq:S-diffusive-upper-bound}.
\end{proof}

Lower bounds are 
more difficult to obtain, and
most of the work for
Theorem~\ref{thm:S-bounds} is needed for the `$\liminf$' part.
Here one can construct a
proof along the lines of the excursion approach from~\cite{hmw1} (see also ~\cite[\S\S 3.7--3.9]{mpw}), but in the present context a related, but more direct approach is available. The following lemma is the main ingredient.

\begin{lemma}
\label{lem:many-big-values}
Let $S := ( S_n, n \in \ZP)$ be an 
$\RP$-valued process adapted to a filtration $(\cF_n, n \in \ZP)$. 
Suppose that
there exist $B, x_0 \in \RP$   and $\delta >0$ such that~\eqref{eq:S-jump-bound}--\eqref{eq:S-zero-drift} hold. Then, for every $\eps >0$, for all $n \in \N$ sufficiently large,
\[ \Pr \left[   \sum_{m = \lfloor n^{1-\eps} \rfloor}^{n} \1 { S_m \geq n^{(1/2)-\eps} } \geq n^{1-\eps}
\right]  \geq 1 - n^{-\eps}. \]
\end{lemma}

To prove this, we need the following relative of Kolmogorov's `other' inequality.

\begin{lemma}
    \label{lem:kolmogorov}
Let $S := ( S_n, n \in \ZP)$ be an 
$\RP$-valued process adapted to a filtration $(\cF_n, n \in \ZP)$. 
Suppose that
there exist $B, x_0 \in \RP$   and $\delta >0$ such that~\eqref{eq:S-jump-bound}--\eqref{eq:S-zero-drift} hold. 
Then, for all $x \in \RP$ and all $n \in\ZP$,
\begin{equation}
    \label{eq:kolmogorov}
    \Pr \Bigl( \max_{n \leq m \leq 2n} S_m \geq x \Bigr)
    \geq 1 - \frac{(x+B)^2}{\delta^3 n} . \end{equation}
\end{lemma}
\begin{proof}
We use a variation on an argument that yields the (sub)martingale version of Kolmogorov's `other' inequality
(see Theorem 2.4.12 of~\cite[p.~45]{mpw}).
Fix $n \in \ZP$, $x \in \RP$,
and let $\sigma_{n,x} := \inf \{ m \in \ZP : S_{n+m} \geq x\}$.
Then from~\eqref{eq:S-sq-lower} we have that
\[ \Exp [ S_{n+ (m+1) \wedge \sigma_{n,x}}^2 -  S_{n+m \wedge \sigma_{n,x}}^2 \mid \cF_{n +  m \wedge \sigma_{n,x}} ] \geq \delta^3 \1 { \sigma_{n,x} > m}.
\]
Taking expectations  and then summing from $m=0$ to $m=k-1$ we obtain
\[ 
\delta^3 \sum_{m=0}^{k-1} \Pr (  \sigma_{n,x} > m  )
\leq \Exp \bigl[ S_{n+k \wedge \sigma_{n,x}}^2 -  S_{n}^2 \bigr]
\leq ( x + B)^2 ,\]
using the fact that $0 \leq S_{n+k \wedge \sigma_{n,x}} \leq \max( S_n , x+B)$, a.s., 
by~\eqref{eq:S-jump-bound}.
It follows that
\[ \Exp ( \sigma_{n,x} ) = \lim_{k \to \infty} \sum_{m=0}^{k-1} \Pr (  \sigma_{n,x} > m  )
\leq \delta^{-3} ( x + B)^2 , \text{ for all } n \in \ZP.
\]
For every $n \in \ZP$ and $x \in \RP$, the event $\sigma_{n,x} \leq n$
implies that $\max_{n \leq m \leq 2n} S_m \geq x$.
Then, by Markov's inequality,
 we obtain, for all $x \in \RP$ and all $n \in\ZP$,
\[
    \Pr \Bigl( \max_{n \leq m \leq 2n} S_m \geq x \Bigr)
    \geq 1 - \Pr ( \sigma_{n,x} \geq n ) 
    \geq 1 - \frac{(x+B)^2}{\delta^3 n} , \]
    as claimed at~\eqref{eq:kolmogorov}.
    \end{proof}

    \begin{proof}[Proof of Lemma~\ref{lem:many-big-values}]
Fix $\eps \in (0,\frac{1}{5})$ and define, for $n \in \N$, $N_\eps (n):= \lfloor n^{1-\eps} \rfloor$, and  
\[ T_n := N_\eps(n) + \inf \{ m \in \{0,1,\ldots, N_\eps (n)\} : S_{N_\eps(n) + m} \geq 2 ( N_\eps (n) )^{(1/2)-\eps} \}, \]
with the usual convention that $\inf \emptyset := \infty$.
Then, by~\eqref{eq:kolmogorov},
\begin{align}
\label{eq:T_n_bound}
    \Pr ( T_n < \infty ) =  \Pr \Bigl( \max_{N_\eps(n)  \leq m \leq 2 N_\eps (n) } S_m \geq 2 ( N_\eps (n) )^{(1/2)-\eps} \Bigr)
& \geq 1 - \frac{(2 ( N_\eps (n) )^{(1/2)-\eps} +B)^2}{\delta^3 N_\eps (n)} \nonumber\\
& \geq 1 - n^{-\eps} ,
\end{align}
for all $n$ sufficiently large. 
For $n \in \N$
large enough so that $( N_\eps (n) )^{(1/2)-\eps} > x_0$, 
we have $S_{T_n} > x_0$ on the event $T_n < \infty$. Define $\tau_n := \inf \{ m \in \ZP : S_{T_n + m} \leq x_0 \}$, and consider
\[ M_{n,m} := ( S_{T_n + m \wedge \tau_n} - S_{T_n} ) \1 { T_n < \infty}, \text{ for } m \in \ZP.
\]
Note that, by~\eqref{eq:S-jump-bound},
we have $|M_{n,m}| \leq B m$, a.s., for all $m \in \ZP$.
Moreover,
\begin{align*}
    \Exp ( M_{n,m+1} - M_{n,m} \mid \cF_{T_n+m} )
& = \Exp ( M_{n,m+1} - M_{n,m} \mid \cF_{T_n+m} ) \1 {T_n < \infty, \,  m < \tau_n}\\
& = \Exp ( S_{T_n+m+1} -  S_{T_n+m}  \mid \cF_{T_n+m} ) \1 {T_n < \infty, \,  m < \tau_n} = 0,
\end{align*}
 by~\eqref{eq:S-zero-drift} and the fact that $m < \tau_n$. Hence $M_{n,m}$, $m \in \ZP$,
is a martingale adapted to $\cF_{T_n+m}$. 
Then, by the martingale property and~\eqref{eq:S-jump-bound},
\[
\Exp ( M_{n,m+1}^2 - M_{n,m}^2 \mid \cF_{T_n+m} ) = 
 \Exp ( ( M_{n,m+1} - M_{n,m} )^2 \mid \cF_{T_n+m} )   \leq  B^2, \as, 
\]
which,  with the fact that $M_{n,0} =0$, implies that
$\Exp ( M_{n,m}^2 \mid \cF_{T_n} ) \leq B^2 m$.
Hence, by Doob's
inequality (e.g.~\cite[p.~35]{mpw}) applied to the non-negative submartingale
$M_{n,m}^2$,  
\begin{equation}
    \label{eq:doob-M}
\Pr \Bigl( \max_{0 \leq m \leq n^{1-5\eps}} M_{n,m}^2 \geq  n^{1- 4\eps} \Bigmid \cF_{T_n} \Bigr)
\leq \frac{\Exp (M_{n,\lfloor n^{1-5\eps} \rfloor}^2 \mid \cF_{T_n})}{n^{1-4\eps}}
\leq n^{-\eps},  \as, 
\end{equation}
for all $n \geq n_0$ for some sufficiently large (deterministic) $n_0 \in \N$. 
 It follows that
 \begin{align*}
 & \Pr \left[ \{ T_n < \infty \} \cap \left\{ \max_{0 \leq m \leq n^{1-5\eps}} | M_{n,m} | \leq   ( N_\eps (n) )^{(1/2) - \eps} \right\} \right] \\
& \geq \Exp \left[ \1 { T_n < \infty }
\Pr \Bigl( \max_{0 \leq m \leq n^{1-5\eps}} M_{n,m}^2 \leq n^{1- 4\eps} \Bigmid \cF_{T_n} \Bigr)
\right] \\
& \geq (1 - n^{-\eps} ) \Pr ( T_n < \infty ) \geq 1 - 2 n^{-\eps},
 \end{align*}
 for all $n$ sufficiently large,
 by~\eqref{eq:T_n_bound}, \eqref{eq:doob-M},
 and the fact that $( N_\eps (n) )^{1 - 2\eps} > n^{1-4\eps}$
 for all large enough~$n$.
If $T_n < \infty$ and $\max_{0 \leq m \leq n^{1-5\eps}} | M_{n,m} | \leq   ( N_\eps (n) )^{(1/2) - \eps}$ both occur, then 
 \[ S_{T_n +m} \geq S_{T_n} - \max_{0 \leq k \leq n^{1-5\eps}} | M_{n,k} |
 \geq   ( N_\eps (n) )^{(1/2) - \eps} , \text{ for all } m \in \{0,1,\ldots, \lfloor n^{1-5\eps} \rfloor \} .\]
In particular, it follows that, for all $n$ large enough
\[ \Pr \left[ \sum_{m=N_\eps (n)}^{n} \1 { S_m \geq   ( N_\eps (n) )^{(1/2)-\eps} } \geq n^{1-6\eps} \right] \geq 1 - 2n^{-\eps} .\]
Since $N_\eps (n) \geq n^{1-2\eps}$, for large $n$, and $\eps>0$ was arbitrary, this concludes the proof.
\end{proof}

\begin{proof}[Proof of Theorem~\ref{thm:S-bounds}]
The `$\limsup$' half of~\eqref{eq:S-diffusive} is given by~\eqref{eq:S-diffusive-upper-bound}.
Moreover, from~\eqref{eq:S-diffusive-upper-bound} we have that,
for any $\eps>0$, $\max_{0 \leq \ell \leq n} S_\ell \leq n^{(1/2)+\eps}$
for all but finitely many $n \in \N$, a.s. In other words, there is a (random) $C_\eps < \infty$
such that $\max_{0 \leq \ell \leq n} S_\ell \leq C_\eps n^{(1/2)+\eps}$ for all $n \in \N$.
Hence, since $\gamma \in \RP$,
\begin{align*}
\Gamma_n (\beta, \gamma) & \leq \sum_{m=1}^n m^{-\beta} \Bigl( \max_{0 \leq \ell \leq m} S_\ell \Bigr)^\gamma  \leq C_\eps \sum_{m=1}^{n} m^{-\beta} m^{(\gamma/2)+\gamma\eps}.
\end{align*}
Suppose $2+\gamma \geq 2\beta$; then for any $\eps>0$ it holds that
$(\gamma/2)+\gamma \eps - \beta > -1$, so that 
$\Gamma_n (\beta, \gamma) \leq C  n^{1+(\gamma/2)+\gamma \eps - \beta}$,
where (random) $C< \infty$, a.s. Since $\eps >0$ can be chosen to be arbitrarily small,
we obtain the `$\limsup$' half of~\eqref{eq:S-limit}. On the other hand, if $2+\gamma < 2 \beta$, then
we can choose $\eps>0$ small enough so that $(\gamma/2)+\gamma \eps - \beta < -1-\eps$,
and we conclude that 
$\Gamma_n (\beta, \gamma) \leq  C_\eps \sum_{m=1}^{n} m^{-1-\eps}$,
from which it follows that $\sup_{n \in \N} \Gamma_n (\beta, \gamma) < \infty$, a.s.

It remains to verify the `$\liminf$' parts of~\eqref{eq:S-diffusive} and~\eqref{eq:S-limit}.
To this end, 
consider event
\[ E_\eps (n) :=
  \sum_{m=\lfloor n^{1-\eps} \rfloor}^{n} \1 { S_m \geq n^{(1/2)-\eps} } \geq n^{1-\eps}.
\]
Then, if $E_\eps (n)$ occurs, we have, since $S_m \geq 0$ for all $m \in \ZP$, and $\gamma \in \RP$,
\begin{align*}
\Gamma_n (\beta , \gamma ) & \geq \sum_{m= \lfloor n^{1-\eps} \rfloor}^n m^{-\beta} S_m^\gamma \1 { S_m \geq n^{(1/2)-\eps} } \\
& \geq n^{(\gamma/2)-\gamma\eps} \Bigl( \min_{\lfloor n^{1-\eps} \rfloor \leq m \leq n}  m^{-\beta} \Bigr)
\sum_{m= \lfloor n^{1-\eps} \rfloor}^n  \1 { S_m \geq n^{(1/2)-\eps} } \\
& \geq n^{\vartheta - \gamma \eps - \eps | \beta | -\eps}, \text{ where } \vartheta := 1 + \frac{\gamma}{2} - \beta.
\end{align*}
Since $\eps>0$ was arbitrary, we conclude
from Lemma~\ref{lem:many-big-values} that, for every $\eps >0$,
\begin{equation}
    \label{eq:Gamma-tail}
    \Pr ( \Gamma_n (\beta , \gamma )  \geq n^{\vartheta -\eps} ) \geq   1 - n^{-\eps} .\end{equation}
From~\eqref{eq:Gamma-tail}
with $n = 2^k$ and an application of 
the Borel--Cantelli lemma we obtain that, a.s., for all but finitely many $k \in \ZP$,
$\Gamma_{2^k} (\beta , \gamma ) \geq (2^{k})^{\vartheta-\eps}$ occurs.
Every $n \in \N$ has $2^{k_n} \leq n < 2^{k_n+1}$ for some $k_n \in \ZP$
with $\lim_{n \to \infty} k_n = \infty$. Hence, since $\Gamma_m (\beta, \gamma)$
is non-decreasing in~$m$, for all but finitely many $n \in \N$,
\[ 
\Gamma_{n} (\beta , \gamma )
\geq \Gamma_{2^{k_n}} (\beta , \gamma )
\geq (2^{k_n})^{\vartheta-\eps}
\geq 2^{-\vartheta} \cdot (2^{k_n+1})^{\vartheta-\eps}
\geq 2^{-\vartheta} n^{\vartheta-\eps}.
\]
Since $\eps>0$ was arbitrary, we obtain the `$\liminf$' part of~\eqref{eq:S-limit}.
Finally, from the $\beta =0$, $\gamma =1$ case of~\eqref{eq:S-limit}
and the fact that we have $\sum_{m=1}^n S_m \leq n \cdot \max_{1 \leq m \leq n} S_m$, we have
\[ \frac{3}{2} = \liminf_{n \to \infty} \frac{ \log \sum_{m=1}^n S_m}{ \log n}
\leq 1 + \liminf_{n \to \infty} \frac{\max_{1 \leq m \leq n} S_m}{ \log n} ,\]
from which  the `$\liminf$' part of~\eqref{eq:S-diffusive} follows.
\end{proof}

\section*{Acknowledgements}
\addcontentsline{toc}{section}{Acknowledgements}

The authors are grateful to  Miha Bre\v sar and  Aleksandar Mijatovi\'c for work 
on a continuum relative of the model considered here,
and to Francis Comets, Iain MacPhee, Serguei Popov, and Stas Volkov
for discussions related to the barycentric excluded volume process. We also thank two
anonymous referees for their suggestions, and corrections, which have improved the paper.
The work of CdC, MM, and AW was supported by EPSRC grant EP/W00657X/1. The work of VS was supported by LMS grant 42248.

\end{document}